\documentclass[11pt,a4paper]{article}

\usepackage{titlesec}
\usepackage{fancyhdr}
\usepackage{a4wide}
\usepackage{graphicx}
\usepackage{float}
\usepackage{amssymb}
\usepackage{amsmath}
\usepackage{amsthm}
\usepackage{color}
\usepackage{mathrsfs}
\usepackage{array}
\usepackage{eucal}
\usepackage{tikz}
\usepackage[T1]{fontenc}
\usepackage{inputenc}
\usepackage[english]{babel}
\usepackage{lmodern}
\usepackage{hyperref}
\usepackage{geometry}
\usepackage{changepage}
\changepage{0pt}{}{}{}{}{0pt}{}{0pt}{10pt}
\usepackage[numbers]{natbib}
\usepackage{hyperref}
\hypersetup{
pdfpagemode=none,
pdftoolbar=true,        
pdfmenubar=true,        
pdffitwindow=false,     
pdfstartview={Fit},    
pdftitle={Non reflection and perfect reflection via Fano resonance in waveguides},    
pdfauthor={L. Chesnel, S.A. Nazarov},     
pdfsubject={},  
pdfcreator={L. Chesnel, S.A. Nazarov},   
pdfproducer={L. Chesnel, S.A. Nazarov}, 
pdfkeywords={}, 
pdfnewwindow=true,      
colorlinks=true,       
linkcolor=magenta,          
citecolor=red,        
filecolor=cyan,      
urlcolor=blue           
}

\newcommand{\dsp}{\displaystyle}

\newcommand{\eps}{\varepsilon}

\newcommand{\Om}{\Omega}
\newcommand{\mrm}[1]{\mathrm{#1}}

\newcommand{\Cplx}{\mathbb{C}}
\newcommand{\N}{\mathbb{N}}
\newcommand{\R}{\mathbb{R}}
\newcommand{\Z}{\mathbb{Z}}

\newcommand{\mL}{\mrm{L}}
\newcommand{\mH}{\mrm{H}}

\newcommand{\rcoef}{\mathcal{R}}
\newcommand{\tcoef}{\mathcal{T}}

\newcommand{\tr}{\mrm{tr}}

\newtheorem{theorem}{Theorem}[section]
\newtheorem{lemma}{Lemma}[section]
\newtheorem{remark}{Remark}[section]

\newtheorem{proposition}{Proposition}[section]

\begin{document}

~\vspace{0.0cm}
\begin{center}
{\sc \bf\LARGE  Non reflection and perfect reflection via\\[6pt] Fano resonance in waveguides}
\end{center}

\begin{center}
\textsc{Lucas Chesnel}$^1$, \textsc{Sergei A. Nazarov}$^{2,\,3}$\\[16pt]
\begin{minipage}{0.95\textwidth}
{\small
$^1$ INRIA/Centre de mathématiques appliquées, \'Ecole Polytechnique, Université Paris-Saclay, Route de Saclay, 91128 Palaiseau, France;\\
$^2$ St. Petersburg State University, Universitetskaya naberezhnaya, 7-9, 199034, St. Petersburg, Russia;\\
$^3$ Institute of Problems of Mechanical Engineering, Bolshoy prospekt, 61, 199178, V.O., St. Petersburg, Russia.\\[10pt]
E-mails: \texttt{lucas.chesnel@inria.fr}, \texttt{srgnazarov@yahoo.co.uk}\\[-14pt]
\begin{center}
(\today)
\end{center}
}
\end{minipage}
\end{center}
\vspace{0.4cm}

\noindent\textbf{Abstract.} 
We investigate a time-harmonic wave problem in a waveguide. By means of  asymptotic analysis techniques, we justify the so-called Fano resonance phenomenon. More precisely, we show that the scattering matrix considered as a function of a geometrical parameter $\eps$ and of the frequency $\lambda$ is in general not continuous at a point $(\eps,\lambda)=(0,\lambda^0)$ where trapped modes exist. In particular, we prove that for a given $\eps\ne0$ small, the scattering matrix exhibits a rapid change for frequencies varying in a neighbourhood of $\lambda^0$. We use this property to construct examples of waveguides such that the energy of an incident wave propagating through the structure is perfectly transmitted (non reflection) or perfectly reflected in monomode regime. We provide numerical results to illustrate our theorems.\\


\section{Introduction}\label{Introduction}

We consider a time-harmonic wave problem in a waveguide unbounded in one direction with a non-penetration (Neumann) boundary condition. This problem appears naturally in many fields, for instance in acoustics, in water-waves theory or in electromagnetism. Our original motivation for the present work was to design waveguides such that the energy of an incident wave propagating through the structure is perfectly transmitted or perfectly reflected. This activity has been the topic of intense studies in the recent years in physics. In particular, we refer the reader to the literature concerning so-called Perfect Transmission Resonances (PTRs), see e.g. \cite{Shao94,PoGP99,LeKi01,Zhuk10,MrMK11}. \\
\newline
At low frequency, that is for $0<k<\pi$ in our particular geometry below, only the two piston modes can propagate and non reflection as well as perfect reflection can be characterized very simply. In such a configuration, in order to describe the scattering process of an incident wave propagating in the waveguide, one usually introduces the reflection coefficient $\rcoef$ and the transmission coefficient $\tcoef$. These complex numbers correspond to the far field amplitudes of the reflected field in the input lead and of the total field in the output lead. Due to conservation of energy, we have the relation
\begin{equation}\label{EnergyConservation}
|\rcoef|^2+|\tcoef|^2=1.
\end{equation} 
In this context, at a given frequency, we say that the geometry is non reflecting if $\rcoef=0$. In this case, the energy of the incident field is perfectly transmitted and the reflected field is exponentially decaying in the input lead. On the other hand, we say that the geometry is perfectly reflecting if $\tcoef=0$. In this situation, the energy is completely backscattered. Though the physical literature concerning Perfect Transmission Resonances is luxuriant, there are only few rigorous mathematical proofs of  existence of waveguides where $\rcoef=0$ or $\tcoef=0$. The goal of the present work is to propose an approach based on the use of the so-called Fano resonance to get these particular values for the scattering coefficients.\\
\newline 
The Fano resonance is a classical phenomenon that arises in many situations in physics (see \cite{Fano61} for the seminal paper and \cite{MiFK10,LZMHN10} for recent reviews). For our particular concern, it appears as follows. Assume that the geometry is characterized by a real parameter $\eps$. Below, $\eps$ will the amplitude of a local perturbation of the walls of the waveguide. Assume that trapped modes exist for the scattering problem for $\eps=0$ at the frequency $\lambda=\lambda^0$. We remind the reader that trapped modes are non-zero solutions of the homogeneous problem which are of finite energy (in particular they decay exponentially in the two leads of the waveguide). Then, for $\eps\ne0$ small, the scattering matrix, which is of size $2\times2$ in monomode regime, exhibits a rapid change for $\lambda$ varying in a neighbourhood of $\lambda^0$. In this article, we exploit this rapid change together with symmetry considerations of the geometry to provide examples of waveguides where $\rcoef=0$ or $\tcoef=0$ (exactly). The main part of this work is dedicated to prove rigorously that the scattering matrix is not smooth at $(\eps,\lambda)=(0,\lambda^0)$. Our study shares similarities with \cite{ShVe05,ShTu12,ShWe13,AbSh16}. In these articles, the authors analyse the Fano resonance phenomenon in gratings in electromagnetism, a context close to ours, via techniques of complex analysis and by means of tools of the theory of analytic functions of several variables. We wish to emphasize that we follow a completely different path. We will work with the augmented scattering matrix $\mathbb{S}$, an object which has been introduced in  \cite{NaPl94bis,KaNa02} and which is smooth at $(\eps,\lambda)=(0,\lambda^0)$. From the relation existing between $\mathbb{S}$ and the usual scattering matrix $\mathfrak{s}$, this will allow us to understand the behaviour of $\mathfrak{s}$ at $(\eps,\lambda)=(0,\lambda^0)$. In our approach, we never work with complex $\lambda$, that is we never deal with the meromorphic extension of $\mathfrak{s}$. In particular, we do not use the fact that when the geometry is perturbed from $\eps=0$ to $\eps\ne0$, in general the eigenvalue $\lambda^0$ embedded in the continuous spectrum becomes a complex resonance, that is a pole of the meromorphic extension of $\mathfrak{s}$ \cite{AsPV00,Zwor99}. We think that this alternative method to analyse the Fano resonance  can be of interest. For computation of complex resonances and numerical investigations of the Fano resonance phenomenon in waveguides, we refer the reader to \cite{DKLM07,CaLo07,EMADAD08,HeKo08,HoNa09,HeKN12}; for results concerning the existence of trapped modes associated with eigenvalues embedded in the continuous spectrum, see \cite{Urse51,Evan92,EvLV94,DaPa98,LiMc07,Naza10c,NaVi10}.\\
\newline
Finally, let us mention that there are at least two other approaches to get $\rcoef=0$ or $\tcoef=0$. The first technique, allowing one to impose $\rcoef=0$ has been proposed in \cite{BoNa13,BLMN15} (see also \cite{BoNTSu,BoCNSu,ChHS15,ChNa16} for applications to other problems) and relies on variants of the proof of the implicit functions theorem. It has been adapted in \cite{BoCNSub} where it is shown how to get $\tcoef=1$ (not only $|\tcoef|=1$). In this case, the energy of the incident propagating wave is perfectly transmitted and the field itself has no phase shift when going through the structure. This kind of techniques, mimicking the proof of the implicit functions theorem, have been introduced in \cite{Naza11,Naza13} and in \cite{Naza11c,Naza12,CaNR12,Naza11b} for a different problem. In these works, the authors construct small non necessarily symmetric perturbations of the walls of a waveguide that preserve the presence of an eigenvalue embedded in the continuous spectrum (the eigenvalue is not turned into a complex resonance).\\
\newline
A second approach to get $\rcoef=0$, $\tcoef=1$ but also $\tcoef=0$ has been developed in \cite{ChNPSu,ChPaSu}. It consists in working in a symmetric waveguide with a branch of length $L$ and to perform an asymptotic analysis of the scattering coefficients as $L\to+\infty$. In \cite{ChNPSu,ChPaSu}, it is proved that the symmetry of the geometry provides enough constraints to guarantee from the asymptotic results that we have exactly $\rcoef=0$, $\tcoef=1$ or $\tcoef=0$ for certain $L$.\\
\newline 
The paper is structured as follows. We start with a $\mrm{1D}$ toy problem for which everything is explicit in order to illustrate the Fano resonance phenomenon. In Section \ref{SectionSetting}, we describe the setting of the $\mrm{2D}$ waveguide problem considered in this work. In Section \ref{SectionPerturbed}, from a situation supporting trapped modes, we perturb both the geometry and the frequency. The perturbation is parameterized by a small characteristic number $\eps$. Then we compute an asymptotic expansion of the augmented scattering matrix $\mathbb{S}^{\eps}$ as $\eps$ tends to zero. In Section \ref{SectionFano}, we gather the asymptotic results obtained for $\mathbb{S}^{\eps}$ and using the formula linking $\mathbb{S}^{\eps}$ to the usual scattering matrix $\mathfrak{s}^{\eps}$, we analyse the behaviour of $\mathfrak{s}^{\eps}$ as $\eps\to0$. This allows us to explain the Fano resonance phenomenon. In Section \ref{SectionRTNull}, we focus our attention on the monomode regime (low frequency) and show how to use the Fano resonance to exhibit waveguides where $\rcoef=0$ or $\tcoef=0$. In Section \ref{SectionNumerics}, we provide numerical results to illustrate our theorems. Finally, we give a short conclusion. The main results of this article are Theorems \ref{MainThmAsympto} and \ref{thmRTNull}.

\section{A $\mrm{1D}$ toy problem}\label{SectionToyPb}
\begin{figure}[!ht]
\centering
\begin{tikzpicture}[scale=1.5]
\begin{scope}[shift={(-1,0.1)}]
\draw[->] (3,0.2)--(3.6,0.2);
\draw[->] (3.1,0.1)--(3.1,0.7);
\node at (3.65,0.3){\small $x$};
\node at (3.25,0.6){\small $y$};
\end{scope}
\draw[dashed] (-4.5,0)--(-4,0);
\draw (-4,0)--(1,0);
\draw (0,0)--(0,1);
\node at (-2.5,-0.3){\small $\Om_1$};
\node at (0.25,0.5){\small $\Om_2$};
\node at (0.5,-0.3){\small $\Om_3$};
\node at (0,-0.25){\small $O$};
\draw[fill=gray!80,draw=none](0,0) circle (0.05);
\end{tikzpicture}
\caption{A $\mrm{1D}$ geometry. \label{Domain1D}} 
\end{figure}
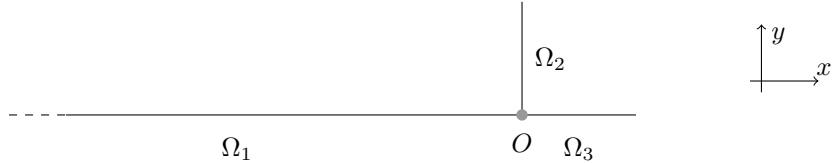

In this section, we consider a $\mrm{1D}$ toy problem to illustrate the Fano resonance. We work in the geometry (see Figure \ref{Domain1D}) 
\[
\Om:=\Om_1\cup\Om_2\cup\Om_3\qquad\mbox{ with }\Om_1:=(-\infty;0)\times\{0\},\ \Om_2:=\{0\}\times(0;1),\,\Om_3:=(0;1)\times\{0\}.
\]
For a function $\varphi$ defined in $\Om$, set $\varphi_i:=\varphi|_{\Om_i}$. Working in suitable coordinates, we can see the $\Om_i$ as $\mrm{1D}$ domains. We consider the Helmholtz problem with Neumann boundary conditions 
\begin{equation}\label{Problem1D}
-\varphi''=k^2\varphi\mbox{ in }\Om,\qquad\qquad\ 
\begin{array}{|l}
\varphi_1(0)=\varphi_2(0)=\varphi_3(0),\\[2pt]
\varphi'_1(0)=\varphi'_2(0)+\varphi'_3(0),\\[2pt]
\varphi'_2(1)=\varphi'_3(1)=0.
\end{array}
\end{equation}
In particular, at the junction point $O$, we impose continuity of the field and conservation of the flux (Kirchhoff law). We are interested in the scattering process of the incident wave $\varphi_i(x)=e^{ikx}$ propagating from $-\infty$. We denote $\varphi$ and $\varphi_s=\varphi-\varphi_i$ the corresponding total and  scattered fields. We impose that $\varphi_s$  is outgoing at infinity. For the simple problem considered here, the radiation condition boils down to assume that $\varphi_s$ writes as $\varphi_s(x)=R\,e^{-ikx}$ for some complex constant $R$ called the reflection coefficient. Using the two boundary conditions of (\ref{Problem1D}), we are led to look for a solution $\varphi$ such that
\[
\varphi_1(x)=e^{ikx}+R\,e^{-ikx},\quad \varphi_2(y)=A\cos(k(y-1)),\quad \varphi_3(x)=B\cos(k(x-1)),
\]
where $A$, $B\in\Cplx$. Writing the transmission conditions at the junction point $O$, we obtain that $R$, $A$, $B$ must solve the system
\begin{equation}\label{system33}
\mathbb{M}\Phi=F\qquad\mbox{ with }\mathbb{M}:=\left(\begin{array}{ccc}
1 & -\cos(k) & 0\\
0 & \cos(k) & -\cos(k)\\
i & \sin(k) & \sin(k)
\end{array}\right),\  \Phi:=\left(\begin{array}{c}
R \\
A \\
B
\end{array}\right)\  F:=\left(\begin{array}{c}
-1 \\
0 \\
i
\end{array}\right).
\end{equation}
One finds that this system (and so Problem (\ref{Problem1D}) with the above mentioned radiation condition) is uniquely solvable if and only $k\not\in (2\N+1)\pi/2$. When $k\in (2\N+1)\pi/2$, the kernel of Problem (\ref{Problem1D}) coincides with $\mrm{span}(\varphi_{\mrm{tr}})$ where $\varphi_{\mrm{tr}}$ is the trapped mode such that $\varphi_{\mrm{tr}}(x)=0$ in $\Om_1$, $\varphi_{\mrm{tr}}(y)=\sin(ky)$ in $\Om_2$ and $\varphi_{\mrm{tr}}(x)=-\sin(kx)$ in $\Om_3$. On the other hand, for any $k>0$, one can check that System (\ref{system33}) (and so Problem (\ref{Problem1D})) admits a solution. Moreover, the coefficient $R$ is always uniquely defined (even when $k\in (2\N+1)\pi/2$) and such that
\begin{equation}\label{DefR1D}
R=\cfrac{\cos(k)+2i\sin(k)}{\cos(k)-2i\sin(k)}.
\end{equation}
The map $k\mapsto R(k)$ is $\pi$-periodic and $|R(k)|=1$. The latter relation, which is due to conservation of energy, guarantees that $R(k)=e^{i\theta(k)}$ for some phase $\theta(k)\in\R/(2\pi\Z)$.\\
\newline
Now we consider the same problem in the perturbed geometry $\Om^{\eps}:=\Om_1\cup\Om_2\cup\Om^{\eps}_3$ with $\Om^{\eps}_3:=(0;1+\eps)\times\{0\}$ and $\eps\in\R$ small. We denote with a superscript $\eps$ all the previously introduced quantities. In $\Om^{\eps}$, the resolution of the previous scattering problem leads to solve the system
\begin{equation}\label{system33eps}
\mathbb{M}^{\eps}\Phi^{\eps}=F\qquad\mbox{ with }\mathbb{M}^{\eps}:=\left(\begin{array}{ccc}
1 & -\cos(k) & 0\\
0 & \cos(k) & -\cos(k(1+\eps))\\
i & \sin(k) & \sin(k(1+\eps))
\end{array}\right),\  \Phi^{\eps}:=\left(\begin{array}{c}
R^{\eps} \\
A^{\eps} \\
B^{\eps}
\end{array}\right).
\end{equation}
The vector $F$ is the same as in (\ref{system33}). For $\eps\ne0$ small, we find that the determinant of $\mathbb{M}^{\eps}$ does not vanish. As a consequence, Problem (\ref{Problem1D}) set in $\Om^{\eps}$ has a unique solution. One finds 
\begin{equation}\label{CoeffReflection1Deps}
R^{\eps}=\cfrac{\cos(k)\cos(k(1+\eps))+i\sin(k(2+\eps))}{\cos(k)\cos(k(1+\eps))-i\sin(k(2+\eps))}.
\end{equation}
Again, we have $|R^{\eps}(k)|=1$ (conservation of energy) so we can write $R^{\eps}(k)=e^{i\theta^{\eps}(k)}$ for some $\theta^{\eps}(k)\in\R/(2\pi\Z)$.  Note that  $\theta^{0}=\theta$ where $\theta$ appears after (\ref{DefR1D}). The map $k\mapsto\theta^{\eps}(k)$ is displayed in Figure \ref{FigPhase} for several values of $\eps$ (see also the alternative representation (\ref{FigFanoSpace1D})). We observe that for $\eps\ne0$, the curve $k\mapsto\theta^{\eps}(k)$ has a fast variation for $k$ close to $\pi/2$. The variation is even faster as $\eps\ne0$ gets small. On the other hand, for $\eps=0$ the curve $k\mapsto\theta^{0}(k)$ has a very smooth behaviour. We emphasize that for $(\eps,k)=(0,\pi/2)$, as mentioned above, trapped modes exist for Problem (\ref{Problem1D}).		
\begin{figure}[!ht]
\centering
\includegraphics[width=0.47\textwidth]{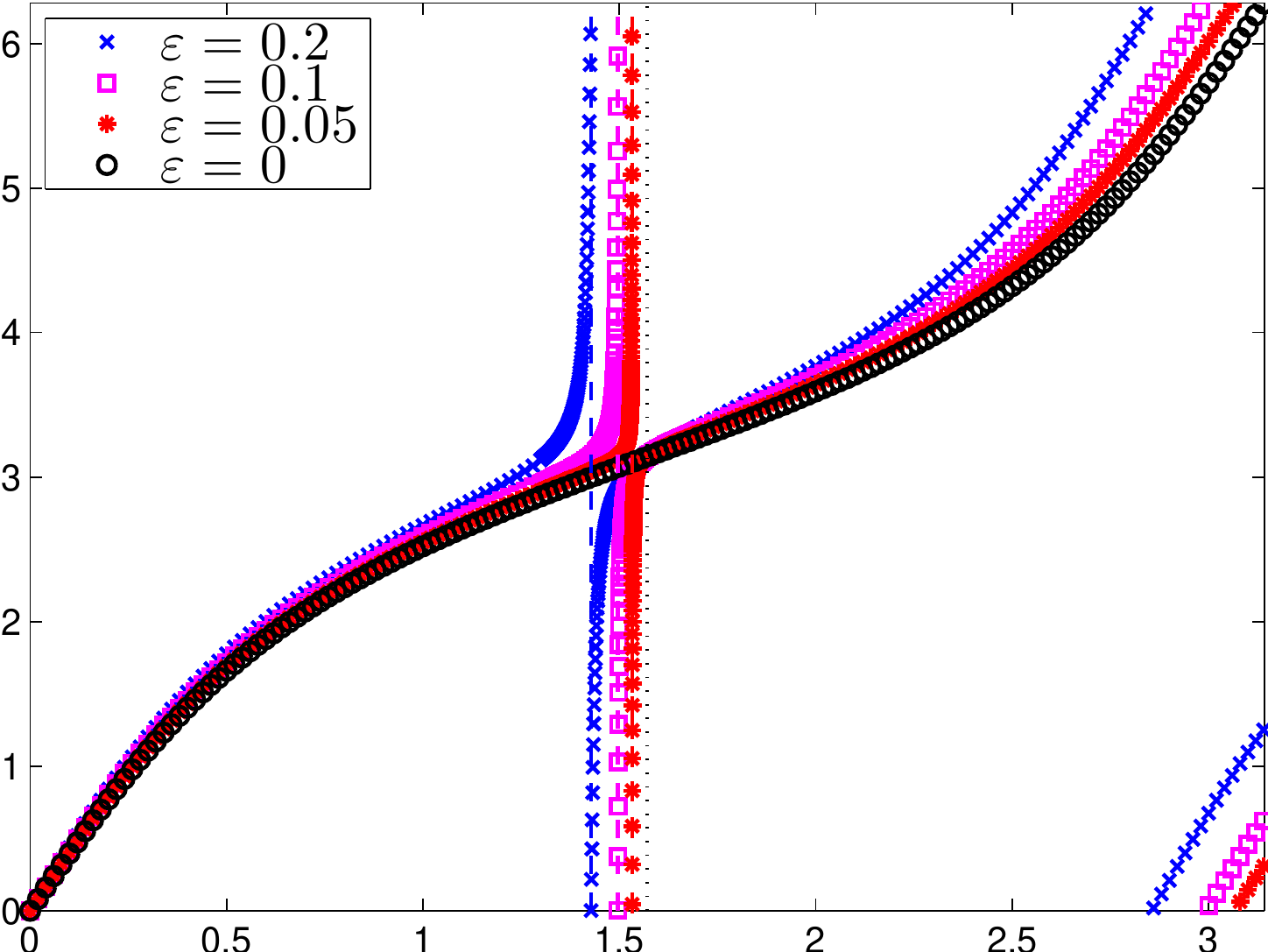}\quad\includegraphics[width=0.47\textwidth]{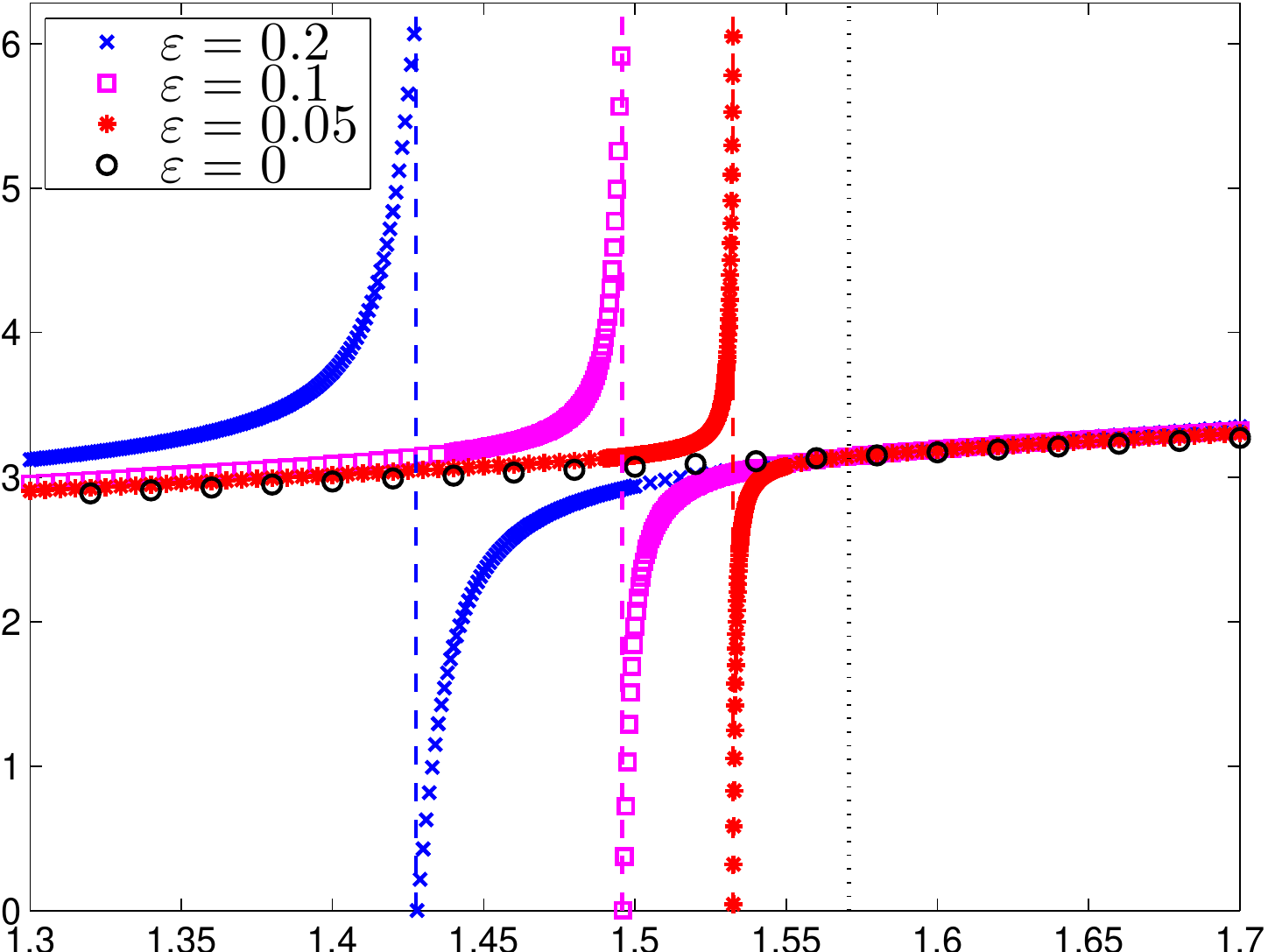}
\caption{Maps $k\mapsto\theta^{\eps}(k)$ for several values of $\eps$. The right picture is a zoom on the left picture around $k=\pi/2$ (marked by the vertical black dotted line). The vertical coloured dashed lines indicate the values of $k$ such that $\theta^{\eps}(k)=0$. \label{FigPhase}}
\end{figure}

\begin{figure}[!ht]
\centering
\includegraphics[width=0.5\textwidth]{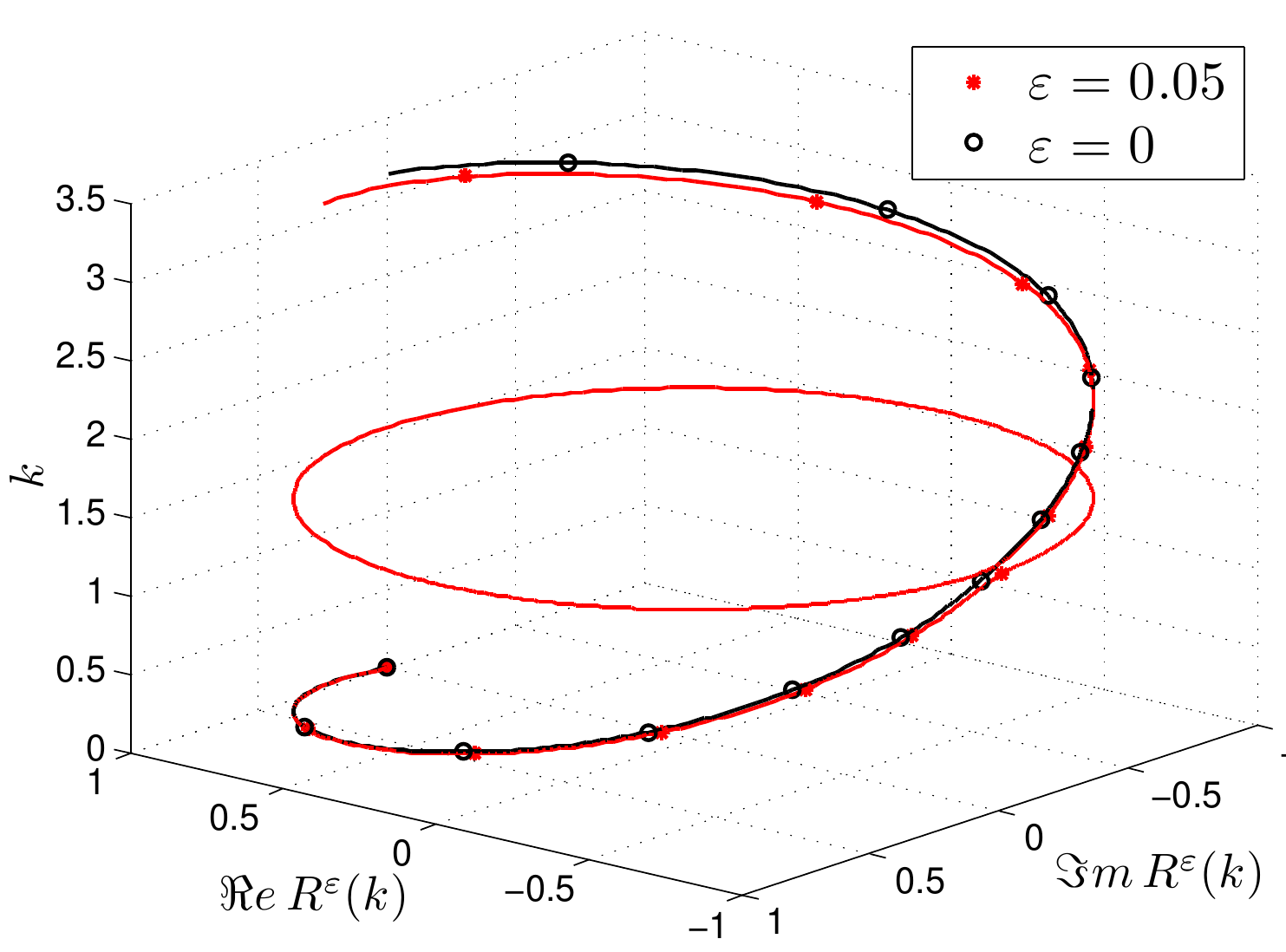}
\caption{Parametric curves $k\mapsto(\Re e\,R^{\eps}(k),\Im m\,R^{\eps}(k))$ for $k\in(0;\pi)$. \label{FigFanoSpace1D}}
\end{figure}
\noindent In order to study the variations of the reflection coefficient with respect to the frequency and the geometry, we define the map $\mathcal{R}:\R^2\to\Cplx$ such that 
\begin{equation}\label{R2variables}
\mathcal{R}(\eps,k)=\cfrac{\cos(k)\cos(k(1+\eps))+i\sin(k(2+\eps))}{\cos(k)\cos(k(1+\eps))-i\sin(k(2+\eps))}.
\end{equation}
With such a notation, we have $R^{\eps}(k)=\mathcal{R}(\eps,k)$ and $R(k)=\mathcal{R}(0,k)$. For all $k\in(0;\pi)$, there holds $\lim_{\eps\to0}\mathcal{R}(\eps,k)=\mathcal{R}(0,k)$. Now assume that the frequency and the geometry are related by some prescribed law in a neighbourhood of the point $(\eps,k)=(0,\pi/2)$ corresponding to a setting supporting trapped modes. For example, assume that $k=\pi/2+\eps k'$ for a given $k'\in\R$. Then for $k'\ne-\pi/4$, starting from expression (\ref{R2variables}), we find as $\eps\to0$ the expansion
\[
\mathcal{R}(\eps,\pi/2+\eps k')=-1+\eps\,\Big(\,\cfrac{-2ik'(\pi+2k')}{\pi+4k'}\,\Big)+O(\eps^2).
\] 
Note that we have $\mathcal{R}(0,\pi/2)=-1$. For $k'=-\pi/4$ and more generally, for $k=\pi/2-\eps(\pi/4)+\eps^2\mu$ with $\mu\in\R$, we obtain
\begin{equation}\label{Resultat1DMobius}
\mathcal{R}(\eps,\pi/2-\eps(\pi/4)+\eps^2\mu)=g(\mu)+O(\eps)\qquad\mbox{ with }g(\mu)=\cfrac{\pi^2+i(32\mu-4\pi)}{\pi^2-i(32\mu-4\pi)}.
\end{equation}
Denote $\mathscr{S}:=\{z\in\Cplx\,|\,|z|=1\}$ the unit circle of the complex plane. Classical results concerning the M\"{o}bius transform (see e.g.  \cite[Chap. 5]{Henr74}) guarantee that $g$ is a bijection between $\R$ and $\mathscr{S}\setminus\{-1\}$. Thus for any $z_0\in\mathscr{S}$, we can find a path $\{\gamma(s),\,s\in(0;1)\}\subset\R^2$ such that $\lim_{s\to1}\gamma(s)=(0,\pi/2)$ and $\lim_{s\to1}\mathcal{R}(\gamma(s))=z_0$ (see Figure \ref{FigFano1DParabola} right). This proves that the map $\mathcal{R}(\cdot,\cdot):\R^2\to\Cplx$ is not continuous at $(0,\pi/2)$. This shows also that for $\eps_0\ne0$ small fixed (see the vertical red dashed line in Figure \ref{FigFano1DParabola} left), the curve $k\mapsto \mathcal{R}(\eps_0,k)$ must exhibit a rapid change. Indeed, for $\mu\in[-C\eps_0^{-1};C\eps_0^{-1}]$ for some arbitrary $C>0$ (which is only a small change for $k$) leads to a large change for $\mathcal{R}(\eps_0,\pi/2-\eps_0(\pi/4)+\eps_0^2\mu)$. This is exactly what we observed in Figure \ref{FigPhase}.  

\begin{figure}[!ht]
\centering
\raisebox{1.3cm}{\begin{tikzpicture}[scale=0.95]
\draw[->] (-2.9,0) -- (3.1,0) node[right] {$\eps$};
\draw[->] (0,-0.2) -- (0,3) node[right] {$k$};
\node at (0,3.14/2-0.1) [left] {$\pi/2$};
\begin{scope}
\clip(-2.5,-0.5) rectangle (2.5,3);
\draw[domain=-2.7:2.7,smooth,variable=\x,blue] plot ({\x},{3.14/2-\x*3.14/4+0.2*\x*\x});
\draw[domain=-2.7:2.7,smooth,variable=\x,blue] plot ({\x},{3.14/2-\x*3.14/4+0.3*\x*\x});
\draw[domain=-2.7:2.7,smooth,variable=\x,blue] plot ({\x},{3.14/2-\x*3.14/4+0.1*\x*\x});
\draw[domain=-2.7:2.7,smooth,variable=\x,blue] plot ({\x},{3.14/2-\x*3.14/4+0.05*\x*\x});
\draw[domain=-2.7:2.7,smooth,variable=\x,blue] plot ({\x},{3.14/2-\x*3.14/4-0.05*\x*\x});
\draw[domain=-2.7:2.7,smooth,variable=\x,blue] plot ({\x},{3.14/2-\x*3.14/4-0.2*\x*\x});
\draw[domain=-2.7:2.7,smooth,variable=\x,blue] plot ({\x},{3.14/2-\x*3.14/4-0.3*\x*\x});
\draw[domain=-2.7:2.7,smooth,variable=\x,blue] plot ({\x},{3.14/2-\x*3.14/4-0.1*\x*\x});
\end{scope}
\draw[red,dashed,very thick] (0.6,-0.2) -- (0.6,2.8);
\node at (0.6,-0.2) [below] {\textcolor{red}{$\eps_0$}};
\end{tikzpicture}}\qquad\includegraphics[width=0.47\textwidth]{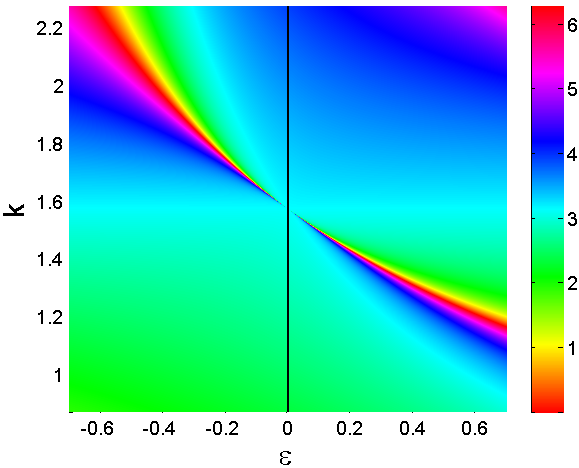}
\caption{Left: several parabolic paths $\{\gamma(s),\,s\in(0;1)\}\subset\R^2$ such that $\lim_{s\to1}\gamma(s)=(0,\pi/2)$. According to the path, the limit of the coefficient $\mathcal{R}(\gamma(s))$ defined in (\ref{R2variables}) as $s\to1$ is different. Right: the colours indicate the phase of $\mathcal{R}(\eps,k)$. This phase is valued in $[0;2\pi)$.\label{FigFano1DParabola}}
\end{figure}
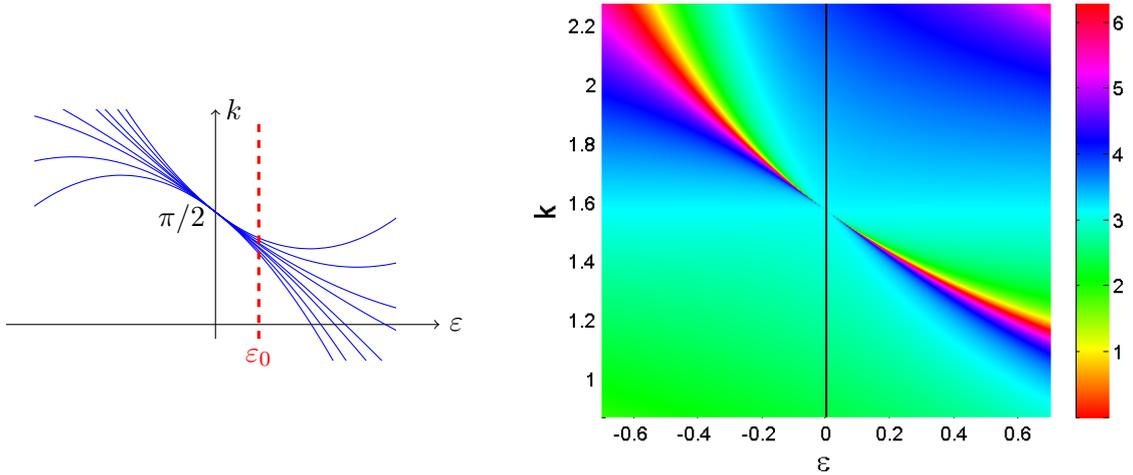

\newpage

\section{Setting}\label{SectionSetting}

~\vspace{-0.8cm}
\begin{figure}[!ht]
\centering
\begin{tikzpicture}[scale=1.5]
\draw[fill=gray!30,draw=black](-1.5,2.5) circle (0.5);
\draw[fill=gray!30,draw=none](-2.5,1) rectangle (-0.5,2.5);
\draw[fill=gray!30,draw=none](-2,1) rectangle (2,2);
\draw[fill=gray!30,draw=none](-5,1) rectangle (-2,2);
\draw (-5,2)--(-2.5,2)--(-2.5,2.5)--(-2,2.5);
\draw (-1,2.5)--(-0.5,2.5)--(-0.5,2)--(-0.5,2) --(2,2);
\draw (-5,1)--(2,1); 
\draw[dashed] (3,1)--(2,1); 
\draw[dashed] (3,2)--(2,2);
\draw[dashed] (-6,2)--(-5,2);
\draw[dashed] (-6,1)--(-5,1);
\draw[->] (3,1.2)--(3.6,1.2);
\draw[->] (3.1,1.1)--(3.1,1.7);
\node at (3.65,1.3){\small $x$};
\node at (3.25,1.6){\small $y$};
\node at (-4,1.2){\small $\Om$};
\draw[dashed,line width=0.5mm,gray!80] (-1.5,1)--(-1.5,3.8);
\node[above] at (0,0.5){\small $d$};
\draw (0,0.9)--(0,1.1);
\node[above] at (-3,0.5){\small $-d$};
\draw (-3,0.9)--(-3,1.1);
\end{tikzpicture}
\caption{Example of geometry $\Om$. The vertical thick dashed line marks the axis of symmetry of the domain. \label{DomainOriginal2D}} 
\end{figure}
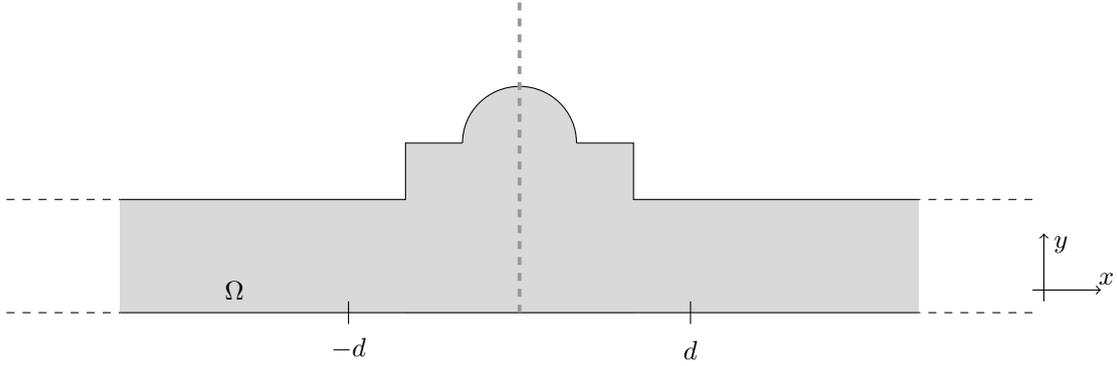

In this section, we introduce notation as well as classical objects and tools that will be used in the rest of the paper. Consider a connected open set $\Om\subset\R^2$ with Lipschitz boundary  $\partial\Om$ which coincides with the reference strip 
\[
\{ (x,y)\in\R\times(0;1)\}
\]
outside a given ball centered at $O$ of radius $d>0$ (see Figure \ref{DomainOriginal2D}). To simplify the presentation, we assume that $\Om$ is symmetric with respect to the $(Oy)$ axis: $\Om=\{(-x,y)\,|\,(x,y)\in\Om\}$. We emphasize that all the analysis we develop can be written in non symmetric waveguides as well. The assumption of symmetry is introduced here only to simplify the presentation. We work in an academic $\mrm{2D}$ waveguide but everything could be done in higher dimension and in a more complex geometry. We assume that the propagation of time-harmonic waves in $\Om$ is governed by the Helmholtz equation with Neumann boundary conditions
\begin{equation}\label{PbInitial}
\begin{array}{|rcll}
\Delta u + \lambda u & = & 0 & \mbox{ in }\Om\\[3pt]
 \partial_{\nu}u  & = & 0  & \mbox{ on }\partial\Om. 
\end{array}
\end{equation}
In this problem, $u$ is the physical field (acoustic pressure, velocity potential, component of the electromagnetic field,...), $\Delta$ denotes the 2D Laplace operator, $\lambda$ is a parameter which is proportional to the square of the frequency and $\nu$ stands for the normal unit vector to $\partial\Om$ directed to the exterior of $\Om$. We assume that the geometry is such that $\lambda=\lambda^0\in(\pi^2(m-1)^2;\pi^2m^2)$, $m\in\N^{\ast}:=\N\setminus\{0\}$, is a simple eigenvalue of the Laplace operator with homogeneous Neumann boundary conditions. In other words, we assume that Problem (\ref{PbInitial}) admits a non zero solution $u_{\tr}$ of finite energy (in $\mL^2(\Om)$) and that any $\mL^2$ solution of (\ref{PbInitial}) is proportional to $u_{\tr}$. Note that in $\Om$ the continuous spectrum of the Neumann Laplacian is equal to $\sigma_c=[0;+\infty)$. As a consequence, we assume here the existence of trapped modes associated with an eigenvalue embedded in the continuous spectrum. In view of symmetry, we observe two options. Either $u_{\tr}$ is symmetric with respect to the $(Oy)$ axis and $u_{\tr}$ satisfies $\partial_xu_{\tr}=0$ on $\Upsilon:=\{(x,y)\in\Om\,|\,x=0\}$. Or $u_{\tr}$ is skew-symmetric with respect to the $(Oy)$ axis and there holds $u_{\tr}=0$ on $\Upsilon$.

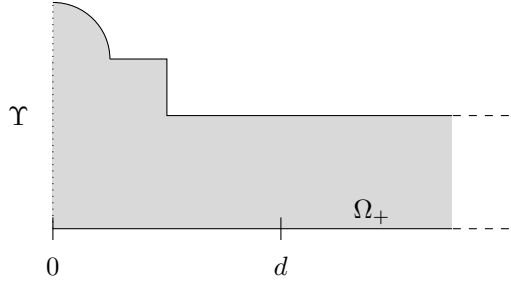
\begin{figure}[!ht]
\centering
\begin{tikzpicture}[scale=1.5]
\draw[fill=gray!30,draw=none](0,2.5)--(0.5,2.5) arc (0:90:0.5);
\draw[fill=gray!30,draw=black](0.5,2.5) arc (0:90:0.5);
\draw[fill=gray!30,draw=none](0,1) rectangle (1,2.5);
\draw[fill=gray!30,draw=none](0,1) rectangle (3.5,2);
\draw (0.5,2.5)--(1,2.5)--(1,2)--(1,2) --(3.5,2);
\draw (0,1)--(3.5,1); 
\draw[dashed] (4,1)--(3.5,1); 
\draw[dashed] (4,2)--(3.5,2);
\node at (2.8,1.15){\small $\Om_+$};
\node[above] at (0,0.5){\small $0$};
\node[above] at (2,0.5){\small $d$};
\draw (2,0.9)--(2,1.1);
\draw (0,0.9)--(0,1.1);
\node at (-0.3,2){\small $\Upsilon$};
\draw[dotted] (0,1)--(0,3);
\end{tikzpicture}
\caption{Half-waveguide $\Om_+$. \label{HalfWaveguide}} 
\end{figure}

\noindent In the following, we restrict the study of (\ref{PbInitial}) in the half-waveguide $\Om_{+}:=\{(x,y)\in\Om\,|\,x>0\}$ (see Figure \ref{HalfWaveguide}). Then we are led to consider the problem  
\begin{equation}\label{PbInitialHalf}
\begin{array}{|rcll}
\Delta u + \lambda u & = & 0 & \mbox{ in }\Om_+\\[3pt]
 \partial_{\nu}u  & = & 0  & \mbox{ on }\partial\Om\cap\partial\Om_+\\[3pt]
\mrm{ABC}(u) & = & 0  & \mbox{ on }\Upsilon.
\end{array}
\end{equation}
In (\ref{PbInitialHalf}), $\mrm{ABC}$ stands for ``Artificial Boundary Condition'' and corresponds to Neumann or Dirichlet boundary conditions according to the symmetry of the function $u_{\tr}$. Clearly $u_{\tr}$ is also a trapped mode for Problem (\ref{PbInitialHalf}) with $\lambda=\lambda^0$ (we make no distinction between $u_{\tr}$ and its restriction to $\Om_+$). To set ideas, we assume that $\|u_{\tr}\|_{\mL^2(\Om_+)}=1$. Using decomposition in Fourier series, we obtain the expansion 
\begin{equation}\label{DefTrappedMode}
u_{\tr}= K\,e^{-\sqrt{\pi^2m^2-\lambda^0} x}\cos(\pi m y)+\tilde{u}_{\tr}\quad\mbox{ for } x\ge d,
\end{equation}
where $K$ is a constant and where $\tilde{u}_{\tr}$ is a remainder which decays as $O(e^{-\sqrt{\pi^2(m+1)^2-\lambda^0} x})$ when $x\to+\infty$. We assume that $u_{\tr}$ has a slow decay, i.e. we assume that there holds $K\ne0$. In case $K=0$, the analysis below must be adapted but can be done. 
Multiplying $u_{\tr}$ by a constant, $K$ can be fixed such that $K>0$.\\
\newline
\noindent Though we assumed the existence of trapped modes in $\Om_+$ at the given frequency, we wish to stress out that all the objects we define up to the end of this section are general objects that are unrelated to the existence of $u_{\tr}$. For $j=0,\dots,m-1$, we define the propagating modes 
\begin{equation}\label{DefModes1}
w_j^{\pm}(x,y)=(a_j)^{-1/2}\,e^{\pm i\alpha_j x}\cos(\pi jy)\quad \mbox{with}\quad \alpha_j=\sqrt{\lambda-\pi^2j^2}\quad\mbox{and}\quad a_j=\begin{array}{|ll}
2\alpha_j & \mbox{ for }j=0\\
\alpha_j & \mbox{ for }j\ne0.
\end{array}
\end{equation}
Note that the normalization constants are chosen so that relations (\ref{SymplecticFormModes}) below hold. Physically, the $w_j^{+}$ (resp. the $w_j^{-}$) are waves which propagate to $+\infty$ (from $+\infty$). Set also 
\[
v_m^{\pm}(x,y)=(a_m)^{-1/2}e^{\pm\alpha_mx}\cos(\pi my)\quad \mbox{with}\quad \alpha_m=\sqrt{\pi^2m^2-\lambda}\quad\mbox{and}\quad a_m=\alpha_m,
\]
and from $v_m^{\pm}$, define the wave packets
\begin{equation}\label{DefModes2}
w_m^{\pm}(x,y)=(2)^{-1/2}\,(v_m^{+}(x,y)\mp iv_m^{-}(x,y)).
\end{equation} 
Remark that $w_m^{\pm}$ are combinations of exponentially growing and decaying modes as $x\to+\infty$. These maybe rather unusual objects will serve to define a scattering matrix, namely the augmented scattering matrix in (\ref{defScatteringMatrix}) (see \cite{NaPl94bis,KaNa02} and \cite{Naza11,Naza12}), which differs from the usual scattering matrix and which will allow us to detect the presence of trapped modes as $u_{\tr}$ with a non zero constant $K\ne0$ in the expansion (\ref{DefTrappedMode}). All the $w^{\pm}_j$, $j=0,\dots,m$, and the  $v_m^{\pm}$ are solutions of the Helmholtz equation in the reference strip $\R\times(0;1)$ and satisfy the homogeneous Neumann boundary condition at $y\in\{0,1\}$. In view of the forthcoming analysis, we introduce a useful tool, namely the symplectic (sesquilinear and anti-hermitian ($q(u,v)=-\overline{q(v,u)}$)) form $q(\cdot,\cdot)$ defined by 
\begin{equation}\label{DefSymplecticForm}
q(u,v)=\int_{0}^1 \partial_x u(d,y)\overline{v}(d,y)-u(d,y)\partial_x\overline{v}(d,y)\,dy.
\end{equation}
With this definition, one can check that we have the orthogonality and normalization conditions
\begin{equation}\label{SymplecticFormModes}
q(w_j^{\pm},w_k^{\pm})=\pm i\delta_{jk}\quad\mbox{ and}\quad q(w_j^{\pm},w_k^{\mp})=0,\quad \mbox{ for }j,k=0,\dots,m,
\end{equation}
where $\delta_{jk}$ stands for the Kronecker symbol. In the geometry $\Om_+$, they are the scattering and augmented scattering solutions
\begin{eqnarray}
\label{scatteringSolutions}\zeta_j&=&w_j^{-}+\dsp\sum_{k=0}^{m-1}s_{kj}\,w_k^{+}+\tilde{\zeta}_j,\qquad j=0,\dots,m-1\\[4pt]
\label{AugscatteringSolutions}Z_j&=&w_j^{-}+\dsp\sum_{k=0}^{m}S_{kj}w_{k}^{+}+\tilde{Z}_j,\qquad j=0,\dots,m.
\end{eqnarray}
Here the $s_{kj}$, $S_{kj}$ are complex constants while the $\tilde{\zeta}_j$, $\tilde{Z}_j$ are remainders which decay respectively as $O(e^{-\sqrt{\pi^2m^2-\lambda} x})$ and as $O(e^{-\sqrt{\pi^2(m+1)^2-\lambda} x})$ when $x\to+\infty$. From the solutions (\ref{scatteringSolutions}), (\ref{AugscatteringSolutions}), we can define two matrices, namely the \textit{scattering matrix} $\mathfrak{s}$ and the \textit{augmented scattering matrix} $\mathbb{S}$ such that
\begin{equation}\label{defScatteringMatrix}
\mathfrak{s}:=\left(s_{kj}\right)_{0\le k,j\le m-1}\in\mathbb{C}^{m\times m}\qquad\qquad\mathbb{S}:=\left(S_{kj}\right)_{0\le k,j\le m}\in\mathbb{C}^{m+1\times m+1}.
\end{equation}
The matrix $\mathfrak{s}$ is the usual scattering matrix while $\mathbb{S}$ has been introduced in \cite{NaPl94bis,KaNa02,Naza11,Naza12}. Observing that $q(\zeta_k,\zeta_j)=q(Z_k,Z_j)=0$ and using formulas (\ref{SymplecticFormModes}) to compute in a different way the quantities $q(\zeta_k,\zeta_j)$, $q(Z_k,Z_j)$, one can verify that $\mathfrak{s}$, $\mathbb{S}$ are unitary ($\mathfrak{s}\,\overline{\mathfrak{s}}^{\top}=\mrm{Id}_{m}$, $\mathbb{S}\,\overline{\mathbb{S}}^{\top}=\mrm{Id}_{m+1}$) and symmetric ($\mathfrak{s}=\mathfrak{s}^{\top}$, $\mathbb{S}=\mathbb{S}^{\top}$). We emphasize that these matrices are uniquely defined even when trapped modes exist at the given frequency. In the following, to simplify notation, we shall write
\[
\zeta=w^{-}+w^{+}\mathfrak{s}+\tilde{\zeta}
\]
with the rows $\zeta=(\zeta_0,\dots,\zeta_{m-1})$, $w^{\pm}=(w^{\pm}_0,\dots,w^{\pm}_{m-1})$, $\tilde{\zeta}=(\tilde{\zeta}_0,\dots,\tilde{\zeta}_{m-1})$. Set also
\[
\mathbb{S}=\left(\begin{array}{cc}\mathbb{S}_{\bullet\bullet}& \mathbb{S}_{\bullet m}\\
\mathbb{S}_{m \bullet } & \mathbb{S}_{m m}
\end{array}\right)\qquad\mbox{with }\mathbb{S}_{\bullet\bullet}\in\mathbb{C}^{m\times m},\ \mathbb{S}_{\bullet m}=\mathbb{S}_{m\bullet}^{\top}\in\mathbb{C}^{m\times 1}\mbox{ and }\mathbb{S}_{m m}\in\mathbb{C}.
\]
Note that in general, the block $\mathbb{S}_{\bullet\bullet}$ is not equal to $\mathfrak{s}$. More precisely, we have the following important formula (see \cite[Thm. 3]{Naza11})
\begin{equation}\label{FormulaScatteringToAugmented}
\mathfrak{s}=\mathbb{S}_{\bullet\bullet}-\mathbb{S}_{\bullet m}(1+\mathbb{S}_{m m})^{-1}\mathbb{S}_{m\bullet}.
\end{equation}
Observe that (\ref{FormulaScatteringToAugmented}) is also valid when $\mathbb{S}_{m m}=-1$. Indeed, in this case the fact that $\mathbb{S}$ is unitary implies that $\mathbb{S}_{\bullet m}=\mathbb{S}_{m\bullet}^{\top}=0$. More precisely, from this analysis we obtain that $\mathfrak{s}=\mathbb{S}_{\bullet\bullet}$ if and only if $|\mathbb{S}_{m m}|=1$. The augmented scattering matrix is a tool to detect the presence of trapped modes. Indeed, we have the following algebraic criterion (see e.g. \cite[Thm. 2]{Naza11}). 
\begin{lemma}\label{LemmaExistenceTrappedMode}
If $\mathbb{S}_{m m}=-1$, then the function $Z_{m}$ defined in (\ref{AugscatteringSolutions}) is a trapped mode for Problem (\ref{PbInitialHalf}) in $\Om_+$. 
\end{lemma}
\begin{proof}
If $\mathbb{S}_{m m}=-1$, since $\mathbb{S}$ is unitary, then $\mathbb{S}_{\bullet m}=\mathbb{S}_{m\bullet}^{\top}=0$. In such a situation, according to (\ref{AugscatteringSolutions}), we have 
\begin{equation}\label{ExpanTrapped}
Z_m=w_m^{-}-w_m^++\tilde{Z}_m=i\sqrt{2}\,v_m^{-}+\tilde{Z}_m=i(2/\alpha_m)^{1/2}\,e^{-\alpha_m x}\cos(\pi m y)+\tilde{Z}_m.
\end{equation}
This shows that $Z_m\not\equiv0$ belongs to $\mH^1(\Om_+)$. In other words $Z_m$ is a trapped mode. 
\end{proof}
\begin{remark}
In our case, since we have considered a setting (geometry and spectral parameter) such that trapped modes of the form (\ref{DefTrappedMode}) exist, we know that precisely $\mathbb{S}_{m m}=-1$. In particular, comparing (\ref{DefTrappedMode}) and (\ref{ExpanTrapped}), we get
\begin{equation}\label{ExpanTrappedBis}
Z_m=iK^{-1}(2/\alpha_m)^{1/2}\,u_{\tr}.
\end{equation}
\end{remark}
\begin{remark}
Note that $\mathbb{S}_{m m}=-1$ is only a sufficient condition of existence of trapped modes. Indeed the geometry $\Om_+$ can support trapped modes for Problem (\ref{PbInitialHalf}) with $\mathbb{S}_{m m}\ne-1$. In this case, these trapped modes must decay as $O(e^{-\sqrt{\pi^2(m+1)^2-\lambda}x})$ when $x\to+\infty$.
\end{remark}

\section{Perturbation of the frequency and of the geometry}\label{SectionPerturbed}

\begin{figure}[!ht]
\centering
\begin{tikzpicture}[scale=1.5]
\draw[fill=gray!30,draw=none](0,2.5)--(0.5,2.5) arc (0:90:0.5);
\draw[fill=gray!30,draw=black](0.5,2.5) arc (0:90:0.5);
\draw[fill=gray!30,draw=none](0,1) rectangle (1,2.5);
\draw[fill=gray!30,draw=none](0,1) rectangle (3.5,2);
\draw[samples=30,domain=-0.25:0.25,draw=black,fill=gray!30] plot(\x+1.7,{2+0.05*(4*\x+1)^4*(4*\x-1)^4*(4*\x+3)});
\draw (0.5,2.5)--(1,2.5)--(1,2)--(1,2) --(3.5,2);
\draw (0,1)--(3.5,1); 
\draw[dashed] (4,1)--(3.5,1); 
\draw[dashed] (4,2)--(3.5,2);
\node at (2.8,1.15){\small $\Om_+^{\eps}$};
\draw[samples=30,domain=-0.25:0.25,draw=none,fill=gray!30] plot(\x+1.7,{1.99+0.05*(4*\x+1)^4*(4*\x-1)^4*(4*\x+3)});
\draw[samples=30,domain=-0.25:0.25,draw=none,fill=gray!30] plot(\x+1.7,{1.995+0.05*(4*\x+1)^4*(4*\x-1)^4*(4*\x+3)});
\node[above] at (2.35,2.05){\small $\gamma^{\eps}=(x,1+\eps H(x))$};
\draw (0,1)--(0,3);
\end{tikzpicture}
\caption{Example of perturbed half-waveguide $\Om_+^{\eps}$. \label{PerturbedHalfWaveguide}} 
\end{figure}
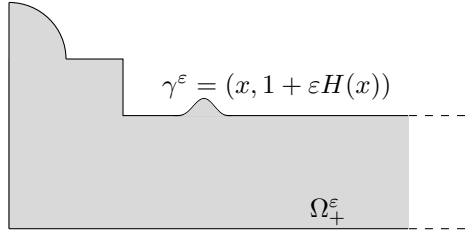

Now, we perturb slightly the original setting supporting trapped modes. More precisely, we make perturbations of amplitude $\eps\in\R$ (small) of the spectral parameter $\lambda^0$ and of the geometry which leads us to consider a family of problems depending on $\eps$. All the objects introduced in the previous section for $\eps=0$ can be defined analogously in the setting depending on $\eps$ and shall be denoted with the superscript $^{\eps}$. Our final goal is to understand the behaviour of the classical scattering matrix $\mathfrak{s}^{\eps}$ (defined as $\mathfrak{s}$ in (\ref{defScatteringMatrix})) as $\eps$ goes to zero. We shall use the following strategy: first we compute an asymptotic expansion of the augmented scattering matrix $\mathbb{S}^{\eps}$ (defined as $\mathbb{S}$ in (\ref{defScatteringMatrix})) as $\eps\to0$ (Section \ref{SectionPerturbed}); then we use Formula (\ref{FormulaScatteringToAugmented}) relating $\mathfrak{s}^{\eps}$ to $\mathbb{S}^{\eps}$ (Section \ref{SectionFano}). The leading idea behind this approach is that the augmented scattering matrix considered as a function of $(\eps,\lambda)$ is smooth at $(0,\lambda^0)$ (see \cite[Lem. 4]{Naza13}).

\subsection{Perturbed setting}\label{paragraphSetting}
In a first step, we explain how we perturb the original setting. The spectral parameter $\lambda^0$ is changed into
\begin{equation}\label{defPerturbFreq}
\lambda^{\eps}=\lambda^0+\eps\lambda'
\end{equation}
where $\lambda'\in\R$ is given and where $\eps>0$ is small. To describe the modification of the geometry $\Om_+$ into $\Om_+^{\eps}$, consider $\gamma\subset\partial\Om_+\setminus\overline{\Upsilon}$ a smooth arc. In a neighbourhood $\mathscr{V}$ of $\gamma$, we introduce natural curvilinear coordinates $(n,s)$ where $n$ is the oriented distance to $\gamma$ such that $n>0$ outside $\Om_+$ and $s$ is the arc length on $\gamma$. Additionally, let $H\in\mathscr{C}^{\infty}_0(\gamma)$ be a smooth profile function which vanishes in a neighbourhood of the two endpoints of $\gamma$. Outside $\mathscr{V}$, we assume that $\partial\Om^{\eps}_+$ coincides with $\partial\Om_+$ and inside $\mathscr{V}$,  $\partial\Om^{\eps}_+$ is defined by the equation
\begin{equation}\label{DefGeom}
n(s)=\eps H(s).
\end{equation}
In other words if $\gamma$ is parametrized as $\gamma=\{P(s)\in\R^2\,|\,s\in I\}$ where $I$ is a given interval of $\R$, then $\gamma^{\eps}:=\{P(s)+\eps H(s)\nu(s)\,|\,s\in I\}$. Here $\nu(s)$ is the normal derivative to $\gamma$ at point $P(s)$ directed to the exterior of $\Om_+$. Thus finally we consider the perturbed problem 
\begin{equation}\label{PbInitialHalfPerturbed}
\begin{array}{|rcll}
\Delta u^{\eps} + \lambda^{\eps} u^{\eps} & = & 0 & \mbox{ in }\Om^{\eps}_+\\[3pt]
 \partial_{\nu^{\eps}}u^{\eps}  & = & 0  & \mbox{ on }\partial\Om^{\eps}_+\setminus\overline{\Upsilon}\\[3pt]
\mrm{ABC}(u^{\eps}) & = & 0  & \mbox{ on }\Upsilon
\end{array}
\end{equation}
where $\nu^{\eps}$ stands for the normal unit vector to $\partial\Om^{\eps}_+$ directed to the exterior of $\Om^{\eps}_+$. Problem (\ref{PbInitialHalfPerturbed}) is the same as Problem (\ref{PbInitialHalf}) in the perturbed setting. Therefore, we could use the notation $(\ref{PbInitialHalfPerturbed})=(\ref{PbInitialHalf})^{\eps}$.
\subsection{Formal expansion}
Following the strategy presented in the introduction of this section, now we compute an asymptotic expansion of the augmented scattering matrix $\mathbb{S}^{\eps}$ as $\eps\to0$. The final results at the end of the procedure are summarized in Proposition \ref{PropositionAsymptoExpan}. For the functions $Z^{\eps}_j$, $j=0,\dots,m$ and the matrix $\mathbb{S}^{\eps}$, we make the following \textit{ans\"{a}tze} 
\begin{eqnarray}
\label{ExpansionScatteringSolutions}Z_j^{\eps}&=&Z_j^{0}+\eps Z_j'+\eps^2Z_j''+\dots\\[3pt]
\label{ExpansionAugScatteringMatrix}\mathbb{S}^{\eps}&=&\mathbb{S}^{0}+\eps\mathbb{S}'+\eps^2\mathbb{S}''+\dots,
\end{eqnarray}
where the dots stand for higher order terms which are not important for our analysis. The next step consists in identifying the terms $Z_j^{0}$, $ Z_j'$, $ Z_j''$ and $\mathbb{S}^{0}$, $\mathbb{S}'$, $\mathbb{S}''$ in the above expansions. To proceed, first we expand the differential operator appearing in the boundary condition of Problem (\ref{PbInitialHalfPerturbed}). We work in the local basis $(\nu(s),\tau(s))$ where $\tau(s)$ is the unit vector orthonormal to $\nu(s)$ such that $(\nu(s),\tau(s))$ is direct. If we denote $\kappa^{\eps}$ the curvature of $\gamma^{\eps}$, which is of order $O(1)$, using that
\[
\begin{array}{rcl}
\nabla\cdot|_{P(s)+n\nu(s)} &=&(\partial_{n}\cdot,(1+n\kappa^{\eps}(s))^{-1}\partial_s\cdot)|_{P(s)+n\nu(s)}\\
\nu^{\eps}(s)&=&(1+\eps^2(1+\eps H(s)\kappa^{\eps}(s))^{-2}|\partial_sH(s)|^2)^{-1/2}\left(1,-\cfrac{\eps \partial_sH(s)}{1+\eps H(s)\kappa^{\eps}(s)}\right),
\end{array}
\]
(here $\nu^{\eps}(s)$ denotes the outward unit normal to $\gamma^{\eps}$ at point $P(s)+\eps H(s)\nu(s)$), we obtain
\[
\partial_{\nu^{\eps}} = (1+\eps^2(1+\eps H\kappa^{\eps}(s))^{-2}|\partial_sH(s)|^2)^{-1/2}(\partial_{n}-\cfrac{\eps \partial_sH(s)}{(1+\eps H\kappa^{\eps}(s))^2}\,\partial_s)=\partial_{n}-\eps \partial_sH(s)\partial_s+\dots
\]
Thus on $\gamma^{\eps}$, plugging (\ref{ExpansionScatteringSolutions}) in the latter expansion, observing that $\partial_n=\partial_{\nu}$ and using that $\partial_{n}Z_j^{0}(\eps H(s),s)=\partial_{n}Z_j^{0}(0,s)+\eps H(s)\partial^2_{n}Z_j^{0}(0,s)+\dots$, we find
\[
\partial_{\nu^{\eps}}Z_j^{\eps}(\eps H(s),s)=\partial_{\nu}Z_j^0(0,s)+\eps\,(\partial_{\nu}Z'_j(0,s)+H(s)\partial^2_nZ^0_j(0,s)-\partial_sH(s)\partial_sZ^0_j(0,s))+\dots.
\]
Finally, plugging (\ref{ExpansionScatteringSolutions}) in Problem (\ref{PbInitialHalfPerturbed}), making $\eps\to0$ and identifying the powers in $\eps$, we first obtain at order $\eps^0$ that $Z_j^{0}$ must be a solution of the problem
\begin{equation}\label{PbInitialHalf0order}
\begin{array}{|rcll}
\Delta Z_j^{0} + \lambda^0 Z_j^{0} & = & 0 & \mbox{ in }\Om_+\\[3pt]
 \partial_{\nu}Z_j^{0}  & = & 0  & \mbox{ on }\partial\Om_+\setminus\overline{\Upsilon}\\[3pt]
\mrm{ABC}(Z_j^{0}) & = & 0  & \mbox{ on }\Upsilon.
\end{array}
\end{equation}
At order $\eps$, for the corrector $Z_j'$, using that $\partial_{n}^2Z^0_j(0,s)=-\partial_s^2Z^0_j(0,s)-\lambda^0 Z^0_j(0,s)$ (because $\partial_{n}Z^0_j(0,s)=0$), we obtain the problem
\begin{equation}\label{PbInitialHalf1order}
\begin{array}{|rcll}
\Delta Z_j'+\lambda^0 Z_j' & = & -\lambda'Z_j^0 & \mbox{ in }\Om_+\\[3pt]
\partial_{\nu}Z_j'(0,s)&=&\partial_s (H(s)\partial_s Z^0_j(0,s))+\lambda^0 H(s)Z^0_j(0,s) & \mbox{ }s\in I\\[3pt]
\partial_{\nu}Z_j'&=&0& \mbox{ on }\partial\Om_+\setminus(\Upsilon\cup\gamma)\\[3pt]
\mrm{ABC}(Z_j') & = & 0  & \mbox{ on }\Upsilon.
\end{array}
\end{equation}
In order to ``close'' systems of equations (\ref{PbInitialHalf0order}), (\ref{PbInitialHalf1order}) to guarantee existence and uniqueness of the solutions, we have to prescribe behaviours for $Z^0_j$, $Z'_j$ as $x\to+\infty$. This will be done in \S\ref{paragraphAtInfinity} studying the expansion of the modes $w_k^{\eps\pm}$ as $\eps\to0$.

\subsection{Behaviours at infinity}\label{paragraphAtInfinity}

In order to compute the behaviours of the modes as $x\to+\infty$ for small $\eps$, we start from the definitions of Section \ref{SectionSetting}. From Formula (\ref{DefModes1}), for $j=0,\dots,m-1$, we have
\[
\alpha^{\eps}_j=\sqrt{\lambda^0+\eps\,\lambda'-\pi^2j^2}=\alpha_j(1+\cfrac{\eps}{2}\,\cfrac{\lambda'}{\alpha_j^2}+O(\eps^2)),\qquad (a^{\eps}_j)^{-1/2}=(a_j)^{-1/2}(1-\cfrac{\eps}{4}\,\cfrac{\lambda'}{\alpha_j^2}+O(\eps^2)).
\]
We deduce 
\begin{equation}\label{expansionModej}
w_j^{\eps\pm}(x,y)=(a^{\eps}_j)^{-1/2}\,e^{\pm i\alpha^{\eps}_j x}\cos(\pi jy)=w_j^{\pm}(x,y)(1+O(\eps(1+|x|))),\quad j=0,\dots,m-1.
\end{equation}
Then we focus our attention on the terms $w_m^{\eps\pm}$. Formula (\ref{DefModes2}) gives
\[
\alpha^{\eps}_m=\sqrt{\pi^2m^2-\lambda^0-\eps\,\lambda'}=\alpha_m(1-\cfrac{\eps}{2}\,\cfrac{\lambda'}{\alpha_m^2}+O(\eps^2)),\qquad (a^{\eps}_m)^{-1/2}=(a_m)^{-1/2}(1+\cfrac{\eps}{4}\,\cfrac{\lambda'}{\alpha_m^2}+O(\eps^2)).
\]
This allows one to write
\begin{equation}\label{expansionModev}
v^{\eps\pm}_m(x,y)=(a^{\eps}_m)^{-1/2}e^{\pm\alpha^{\eps}_mx}\cos(\pi my)=v^{\pm}_m(x,y)(1+\cfrac{\eps}{4}\,\cfrac{\lambda'}{\alpha_m^2}(1\mp2\alpha_mx)+O(\eps^2(1+x^2))).
\end{equation}

\subsection{Definition and computation of the terms in the expansions}
Gathering (\ref{PbInitialHalf0order}) and expansions (\ref{expansionModej}), (\ref{expansionModev}) at order $\eps^0$, we deduce that the functions $Z_j^{0}$ appearing in the \textit{ans\"{a}tze} (\ref{ExpansionScatteringSolutions}) for $Z_j^{\eps}$ must coincide with the $Z_j$ introduced in (\ref{AugscatteringSolutions}), $j=0,\dots,m-1$. As a consequence, in the \textit{ans\"{a}tze} $\mathbb{S}^{\eps}=\mathbb{S}^{0}+\eps\mathbb{S}'+\eps^2\mathbb{S}''+\dots$, we must take $\mathbb{S}^0=\mathbb{S}$ where $\mathbb{S}$ is defined in (\ref{defScatteringMatrix}).\\
\newline
Then, we work at order $\eps$. We focus our attention on the term $Z^{\eps}_m$, this will be enough for our needs. For $\eps>0$, we have the formula
\[
Z^{\eps}_m=w_m^{\eps-}+\dsp\sum_{k=0}^{m}S^{\eps}_{km}w_{k}^{\eps+}+\tilde{Z}^{\eps}_m.
\]
Using that $\mathbb{S}^{\eps}=\mathbb{S}^{0}+\eps\mathbb{S}'+\eps^2\mathbb{S}''+\dots$ with $\mathbb{S}^0=\mathbb{S}$ (and in particular $\mathbb{S}_{m m}=-1$, $\mathbb{S}_{\bullet m}=0$), we get 
\begin{equation}\label{RelInter}
Z^{\eps}_m=w_m^{\eps-}-w_m^{\eps+}+\eps\dsp\sum_{k=0}^{m}S'_{km}w_{k}^{\eps+}+\tilde{Z}^{\eps}_m+O(\eps^2).
\end{equation}
But from $w_m^{\eps\pm}=(2)^{-1/2}\,(v_m^{\eps+}\mp iv_m^{\eps-})$, we infer that $w_m^{\eps-}-w_m^{\eps+}=\sqrt{2} i v_m^{\eps-}$. Plugging the latter relation in (\ref{RelInter}), using (\ref{expansionModev}) and comparing with the expansion $Z^{\eps}_m=Z^{0}_m+\eps Z'_m+\dots$, we obtain 
\begin{equation}\label{ExpansionCorrector}
Z'_m(x,y)=\cfrac{\sqrt{2}i}{4}\,\cfrac{\lambda'}{\alpha_m^2}(1+2\alpha_mx)v^{-}_m(x,y)+\dsp\sum_{k=0}^{m}S'_{km}w_{k}^{+}(x,y)+\tilde{Z}'_m(x,y).
\end{equation}
One can check that Problem (\ref{PbInitialHalf1order}) admits a solution of the form (\ref{ExpansionCorrector}). In the following, we will need to assess more precisely the value of the terms $\mathbb{S}'_{mm}:=S'_{mm}$, $\mathbb{S}'_{\bullet m}:=(S'_{0 m},\dots,S'_{m-1 m})^{\top}$.\\
\newline 
\noindent$\star$ We start with $\mathbb{S}'_{mm}$. Because of the presence of the linear term in (\ref{ExpansionCorrector}), we dot not work with the symplectic form $q(\cdot,\cdot)$. Instead we multiply (\ref{PbInitialHalf1order}) for $j=m$ by $\overline{Z^0_m}=\overline{Z_m}$ and integrate by parts  in the domain $\Om_{+,\,R}:=\{(x,y)\in\Om_+\,|\,|x|<R\}$. Since $Z_m$ is a solution of (\ref{PbInitialHalf0order}) which is exponentially decaying as $x\to+\infty$, taking the limit $R\to+\infty$, we obtain
\begin{equation}\label{IntParParts}
\begin{array}{ll}
&\lambda'\dsp\int_{\Om_+}|Z_m(x,y)|^2\,dxdy+\int_{I}(\partial_s(H(s)\partial_sZ_m(0,s))+\lambda^0 H(s)Z_m(0,s))\overline{Z_m(0,s)}\,ds\\[10pt]
=&-\dsp\lim_{R\to+\infty}\int_{0}^1\partial_x Z'_m(R,y)\overline{Z_m(R,y)}- Z'_m(R,y)\overline{\partial_xZ_m(R,y)}\,dy.
\end{array}
\end{equation}
Then we compute the right hand side of (\ref{IntParParts}) thanks to (\ref{ExpansionCorrector}) and using the identities $Z_m(x,y)=i\sqrt{2}\,v_m^-+\tilde{Z}_m=iK^{-1}\sqrt{2/\alpha_m}u_{\tr}$ (see (\ref{ExpanTrappedBis})) as well as $\|u_{\tr}\|_{\mL^2(\Om_+)}=1$, we find
\begin{equation}\label{defTermCorrectorScamm}
\mathbb{S}'_{mm}=-2i\alpha_m^{-1}K^{-2}(\lambda'-\ell_m(H))\quad\mbox{ with }\ell_m(H)=\int_{I}H(s)(|\partial_su_{\tr}(0,s)|^2-\lambda^0\,|u_{\tr}(0,s)|^2)\,ds.
\end{equation}
$\star$ Now, we focus our attention on the term $\mathbb{S}'_{\bullet m}$. Subtracting $c\,u_{\tr}$ where $c$ is a constant, we can assume that the functions $\zeta_j$ defined in (\ref{scatteringSolutions}) admit the expansion
\[
\zeta_j=w_j^{-}+\dsp\sum_{k=0}^{m-1}s_{kj}\,w_k^{+}+0\,v^{-}_m+\tilde{\zeta}_j,\qquad j=0,\dots,m-1,
\]
where this time $\tilde{\zeta}_j$ decays as $O(e^{-\sqrt{\pi^2(m+1)^2-\lambda^0}x})$ when $x\to+\infty$. Multiplying (\ref{PbInitialHalf1order}) by $\overline{\zeta}=(\overline{\zeta_0},\dots,\overline{\zeta_{m-1}})$ and using row notation leads to 
\[
\begin{array}{ll}
&\lambda'(Z_m,\zeta)_{\Om_+}+\dsp\int_{I}(\partial_s(H(s)\partial_sZ_m(0,s))+\lambda^0 H(s)Z_m(0,s))\overline{\zeta(0,s)}\,ds\\[4pt]
=&-\dsp\lim_{R\to+\infty}\int_{0}^1\partial_x Z'_m(R,y)\overline{\zeta(R,y)}- Z'_m(R,y) \partial_x \overline{\zeta(R,y)}\,dy=-i\,\mathbb{S}'_{m\bullet}\,\overline{\mathfrak{s}},
\end{array}
\]
where $(\cdot,\cdot)_{\Om_+}$ is the usual inner product (sesquilinear) of $\mL^2(\Om_+)$. Since $\mathfrak{s}$ is unitary, we deduce 
\begin{equation}\label{correctorTrans1}
\mathbb{S}'_{m\bullet}=-\sqrt{\cfrac{2}{\alpha_m}}\,K^{-1}(\lambda't-\ell_{\bullet}(H))\,\mathfrak{s}
\end{equation}
with
\begin{equation}\label{DefScndAssumption}
t=\int_{\Om_+}\hspace{-0.2cm}u_{\tr}(x,y)\overline{\zeta(x,y)}\,dxdy\ \mbox{ and }\ \ell_{\bullet}(H)=\int_{I}H(s)(\partial_su_{\tr}(0,s)\overline{\partial_s\zeta(0,s)}-\lambda^0\,u_{\tr}(0,s)\overline{\zeta(0,s)})\,ds.
\end{equation}
Note that due to the symmetry of $\mathbb{S}^{\eps}$, we have $\mathbb{S}'_{\bullet m}=(\mathbb{S}'_{m\bullet})^{\top}$. We summarize the results we have obtained in the proposition below. Its proof can be established following the approach leading to \cite[Thm. 9]{Naza13}. It relies on the use of weighted spaces with detached asymptotics (see \cite{Naza99a,Naza99b,Naza11} and others). In the process, it is necessary to rectify the boundary using ``almost identical'' diffeomorphisms to transform the perturbed waveguide $\Om^{\eps}_+$ into the original geometry $\Om_+$. Then one can use classical approaches to deal with perturbations of linear operators (see e.g. \cite[Chap. 7, \S6.5]{Kato95} or \cite[Chap. 4]{HiPh57}).

\begin{proposition}\label{PropositionAsymptoExpan}
There exists $\eps_0>0$ and $\mathbb{S}''_{mm}\in\Cplx$ such that for all $\eps\in(0;\eps_0]$, 
\begin{equation}\label{ErrorEstimates}
|\mathbb{S}^{\eps}_{mm}-(-1+\eps\mathbb{S}'_{mm}+\eps^2\mathbb{S}''_{mm})| \le C\,\eps^3,\qquad  |\mathbb{S}^{\eps}_{m\bullet}-\eps\mathbb{S}'_{m\bullet}| \le C\,\eps^2\quad\mbox{ and }\quad|\mathbb{S}^{\eps}_{\bullet\bullet}-\mathfrak{s}| \le C\,\eps.
\end{equation}
In (\ref{ErrorEstimates}), $\mathfrak{s}$ is the usual scattering matrix for $\lambda=\lambda^0$ while $\mathbb{S}'_{mm}$, $\mathbb{S}'_{\bullet m}$ are respectively defined in (\ref{defTermCorrectorScamm}), (\ref{correctorTrans1}). The constant $C>0$ can be chosen independently of $\lambda'$ (see (\ref{defPerturbFreq})) in a compact interval of $\R$. 
\end{proposition}
\begin{remark}
We do not compute the value of $\mathbb{S}''_{mm}$ in (\ref{ErrorEstimates}) because we will not need it. However this can be done following the approach above. 
\end{remark}

\section{The Fano resonance}\label{SectionFano}
Now we gather the asymptotic formulas derived for $\mathbb{S}^{\eps}$ in the previous section to analyse the behaviour of $\mathfrak{s}^{\eps}$ as $\eps\to0$ and to explain the Fano resonance. We plug the expansion $\mathbb{S}^{\eps}=\mathbb{S}+\eps\mathbb{S}'+\eps^2\mathbb{S}''+\dots$ in the key formula 
\[
\mathfrak{s}^{\eps}=\mathbb{S}^{\eps}_{\bullet\bullet}-\mathbb{S}^{\eps}_{\bullet m}(1+\mathbb{S}^{\eps}_{m m})^{-1}\mathbb{S}^{\eps}_{m\bullet}
\]
(see (\ref{FormulaScatteringToAugmented})). Using estimate (\ref{ErrorEstimates}), we obtain
\begin{equation}\label{AsymptoticUsualSca}
\mathfrak{s}^{\eps}=\mathfrak{s}+O(\eps)-\eps(\mathbb{S}'_{mm}+\eps\mathbb{S}''_{mm}+O(\eps^2))^{-1}(\mathbb{S}'_{\bullet m}\mathbb{S}'_{m\bullet}+O(\eps)).
\end{equation}
From (\ref{AsymptoticUsualSca}), we see that if $\mathbb{S}'_{mm}\ne0$, which is equivalent to $\lambda'\ne\ell_m(H)$ according to (\ref{defTermCorrectorScamm}), then we have 
\begin{equation}\label{relationAnalysis1}
\mathfrak{s}^{\eps}=\mathfrak{s}+O(\eps).
\end{equation}
In order to express the dependence of the scattering matrix with respect to both $\eps$ and $\lambda$, we define the map $\mathsf{s}:\R^2\to\Cplx^{m\times m}$ such that $\mathsf{s}(\eps,\lambda^{\eps})=\mathfrak{s}^{\eps}$. With such a notation, we have $\mathsf{s}(0,\lambda^0)=\mathfrak{s}$. Relation (\ref{relationAnalysis1}) gives
\[
\lim_{\eps\to0}\ \mathsf{s}(\eps,\lambda^0+\eps\lambda')=\mathfrak{s}\qquad\mbox{ for }\lambda'\ne\ell_m(H).
\]
Now we consider the case $\mathbb{S}'_{mm}=0\Leftrightarrow\lambda'=\ell_m(H)$. Note that (\ref{defTermCorrectorScamm}) guarantees that $\ell_m(H)$ is real. More generally, we assume that $\lambda'=\ell_m(H)+\eps\mu$ with $\mu\in\R$, so that $\lambda^{\eps}=\lambda^0+\eps\ell_m(H)+\eps^2\mu$. Observe the analogy with what has been done in Section \ref{SectionToyPb} (see the discussion after (\ref{R2variables})). Then, using (\ref{defTermCorrectorScamm}) in (\ref{AsymptoticUsualSca}), we obtain the new formula for the usual scattering matrix
\begin{equation}\label{MainResultAsymptotic}
\mathfrak{s}^{\eps}=\mathfrak{s}+\cfrac{\mathbb{S}'_{\bullet m}\mathbb{S}'_{m\bullet}}{2i\alpha_m^{-1}K^{-2}\mu-\mathbb{S}''_{mm}}+O(\eps).
\end{equation}
With the above notation, this implies
\begin{equation}\label{MainResultAsymptoticFormulation2}
\lim_{\eps\to0}\ \mathsf{s}(\eps,\lambda^0+\eps\ell_m(H)+\eps^2\mu)=\mathfrak{s}+\cfrac{\mathbb{S}'_{\bullet m}\mathbb{S}'_{m\bullet}}{2i\alpha_m^{-1}K^{-2}\mu-\mathbb{S}''_{mm}}.
\end{equation}
Note that to derive expansion (\ref{MainResultAsymptotic}), we need to guarantee that the denominator does not vanish. This is true for all $\mu\in\R$ as soon as $\Re e\,\mathbb{S}''_{mm}$ is not null. Let us obtain a necessary and sufficient condition so that this holds. Due to the fact that $\mathbb{S}^{\eps}$ is unitary, we have $1=|\mathbb{S}^{\eps}_{mm}|^2+|\mathbb{S}^{\eps}_{m\bullet}|^2$. Plugging the expansions 
\[
\mathbb{S}^{\eps}_{mm}=-1+\eps\mathbb{S}_{mm}'+\eps^2\mathbb{S}''_{mm}+\dots\qquad\qquad \mathbb{S}^{\eps}_{m\bullet}=\eps\mathbb{S}_{m\bullet}'+\eps^2\mathbb{S}''_{m\bullet}+\dots
\]
in this identity and using that $\mathbb{S}_{mm}'$ is purely imaginary (see again (\ref{defTermCorrectorScamm})), we find
\[
1=|\mathbb{S}^{\eps}_{mm}|^2+|\mathbb{S}^{\eps}_{m\bullet}|^2=1+\eps^2(-2\Re e\,\mathbb{S}''_{mm}+|\Im m\,\mathbb{S}'_{mm}|^2+|\mathbb{S}'_{m\bullet}|^2)+O(\eps^3).
\]
In particular, when $\lambda'=\ell_m(H)+O(\eps)$, using formulas (\ref{defTermCorrectorScamm}) and (\ref{correctorTrans1}), we deduce 
\begin{equation}\label{RealPartScndCorrector}
\Re e\,\mathbb{S}''_{mm}=|\mathbb{S}'_{m\bullet}|^2/2= \alpha_m^{-1}K^{-2}\,|\ell_{\bullet}(H)-\ell_m(H)t|^2.
\end{equation}
Thus, we see that (\ref{MainResultAsymptotic}) (and so (\ref{MainResultAsymptoticFormulation2})) is valid when the profile function $H$ is such that
\begin{equation}\label{AssumptionCoeffs}
\ell_{\bullet}(H)\ne\ell_{m}(H)t\in\mathbb{C}^{m}.
\end{equation}
Note that according to (\ref{RealPartScndCorrector}), the latter assumption also guarantees that $\mathbb{S}'_{\bullet m}\ne0$ and $\mathbb{S}'_{m\bullet}\ne0$. We gather these results in the following theorem, the main outcome of the article. 
\begin{theorem}\label{MainThmAsympto}
Let $\mathsf{s}(\eps,\lambda)$ stand for the usual scattering matrix  (\ref{defScatteringMatrix}) for the problem (\ref{PbInitialHalf}) in $\Om^{\eps}$.\\
$\star$ Assume that $\lambda'\ne\ell_m(H)$ where $\ell_m(H)$ is defined in (\ref{defTermCorrectorScamm}). Then we have
\[
\lim_{\eps\to0}\ \mathsf{s}(\eps,\lambda^0+\eps\lambda')=\mathfrak{s}.
\]
$\star$ Assume that $H$ is such that $\ell_{\bullet}(H)\ne\ell_{m}(H)t\in\mathbb{C}^{m}$ where $\ell_{\bullet}(H)$ and $t$ are defined in (\ref{DefScndAssumption}). Then we have
\[
\lim_{\eps\to0}\ \mathsf{s}(\eps,\lambda^0+\eps\ell_m(H)+\eps^2\mu)=\mathfrak{s}+\cfrac{\mathbb{S}'_{\bullet m}\mathbb{S}'_{m\bullet}}{2i\alpha_m^{-1}K^{-2}\mu-\mathbb{S}''_{mm}}.
\]
\end{theorem}
\noindent Let us make some comments about this result. As for the map $\mathcal{R}(\cdot,\cdot)$ defined in (\ref{R2variables}) for the $\mrm{1D}$ toy problem (see Section \ref{SectionToyPb}), we see that $\mathsf{s}(\cdot,\cdot)$ is not continuous at the point $(0,\lambda^0)$. We have the same picture as in Figure \ref{FigFano1DParabola} left. The function $\mathsf{s}(\cdot,\cdot)$ valued on different parabolic paths $\{(\eps,\lambda^{0}+\eps\lambda'+\eps^2\mu),\,\eps\in(0;\eps_0)\}$ can have different limits as $\eps$ tends to zero. And for $\eps_0\ne0$ small fixed, the usual scattering matrix $\lambda\mapsto\mathsf{s}(\eps_0,\lambda)$ exhibits a quick change in a neighbourhood of $\lambda^0+\eps_0\ell_{m}(H)$. Indeed, the map $\mu\mapsto\mathsf{s}(\eps_0,\lambda^0+\eps_0\ell_m(H)+\eps_0^2\mu)$ has a large variation for $\mu\in[-C\eps_0^{-1};C\eps_0^{-1}]$ for some arbitrary $C>0$ (which is only a small change for $\lambda^{\eps_0}$). Said differently, a change of order $\eps$ of the frequency leads to a change of order one of the scattering matrix. This is the Fano resonance. For a given $C>0$, outside an interval of length $C\eps_0$ centered at $\lambda^0+\eps_0\ell_m(H)$, $\mathsf{s}(\eps_0,\cdot)$ is approximately equal to $\mathfrak{s}$. 
\begin{remark}
When $H$ is such that $\ell_{\bullet}(H)=\ell_{m}(H)t\in\mathbb{C}^{m}$, a fast Fano resonance phenomenon appears. More precisely, for a given $\eps_0\ne0$ small, the rapid variation of $\mathsf{s}(\eps_0,\cdot)$ occurs on a range of frequencies of length $O(\eps_0^2)$ (instead of $O(\eps_0)$ when $\ell_{\bullet}(H)\ne\ell_{m}(H)t\in\mathbb{C}^{m}$).
\end{remark}
\noindent For the monomode regime $m=1$, that is when only $w_{0}^{\pm}$, the piston modes, are propagating modes, the usual scattering matrix $\mathfrak{s}$ is just a complex number of modulus one (due to conservation of energy). In this case, we also have $\mathbb{S}'_{\bullet m}=\mathbb{S}'_{m\bullet }\in\mathbb{C}\setminus\{0\}$. From (\ref{MainResultAsymptotic}), (\ref{RealPartScndCorrector}), we find
\begin{equation}\label{eqn1DAnalysis}
\mathsf{s}(\eps,\lambda^0+\eps\ell_m(H)+\eps^2\mu)=\mathfrak{s}+\cfrac{(\mathbb{S}'_{\bullet m})^2}{i\tilde{\mu}-|\mathbb{S}'_{\bullet m}|^2/2}+O(\eps)\qquad\mbox{with }\ \tilde{\mu}:=2\alpha_m^{-1}K^{-2}\mu-\Im m\,\mathbb{S}''_{m m}.
\end{equation}
The fact that $\mathbb{S}^{\eps}$ is unitary (and symmetric) also implies that $\mathbb{S}_{\bullet\bullet}^{\eps}\overline{\mathbb{S}_{m\bullet}^{\eps}}+\mathbb{S}_{m\bullet}^{\eps}\overline{\mathbb{S}_{mm}^{\eps}}=0$. Plugging the expansions $\mathbb{S}^{\eps}_{\bullet\bullet}=\mathfrak{s}+O(\eps)$, $\mathbb{S}^{\eps}_{m\bullet}=\eps\mathbb{S}_{m\bullet}'+O(\eps^2)$ and $\mathbb{S}^{\eps}_{mm}=-1+O(\eps)$ in this identity, we get $\mathfrak{s}=(\mathbb{S}_{m\bullet}')^2/|\mathbb{S}_{m\bullet}'|^2$. Using the latter relation in (\ref{eqn1DAnalysis}) yields 
\begin{equation}\label{eqn1DAnalysisBis}
\mathsf{s}(\eps,\lambda^0+\eps\ell_m(H)+\eps^2\mu)=\mathfrak{s}\,\cfrac{2i\tilde{\mu}+|\mathbb{S}'_{\bullet m}|^2}{2i\tilde{\mu}-|\mathbb{S}'_{\bullet m}|^2}+O(\eps).
\end{equation}
Formula (\ref{eqn1DAnalysisBis}) is very similar to Expansion (\ref{Resultat1DMobius}) obtained in $\mrm{1D}$. Gathering (\ref{eqn1DAnalysis}) and (\ref{eqn1DAnalysisBis}), we deduce that asymptotically as $\eps$ goes to zero, when the frequency increases ($\mu$ and so $\tilde{\mu}$ varies from $-\infty$ to $+\infty$), the coefficient $\mathsf{s}(\eps,\cdot)$ runs exactly once on the unit circle counter-clockwise from $\mathfrak{s}$ to $\mathfrak{s}$. This is exactly what we found in $\mrm{1D}$ (see Figure \ref{FigFanoSpace1D}).

\section{Connection with non reflection and perfect reflection}\label{SectionRTNull}
We come back to the geometry which is unbounded both in left and right directions introduced in Section \ref{SectionSetting}. We consider again the monomode regime with $m=1$ (this is important in the following analysis) assuming that $\lambda^0\in(0;\pi^2)$. We remind the reader that $\lambda^0$ is a simple eigenvalue of the Laplace operator with homogeneous Neumann boundary conditions in $\Om$. To set ideas, we assume that the corresponding eigenfunctions are symmetric with respect to the $(Oy)$ axis so that there holds $\partial_xu_{\mrm{tr}}=0$ on $\Upsilon$. We denote $\Om^{\eps}$ the domain obtained from $\Om^{\eps}_+$ (see \S\ref{paragraphSetting}) by symmetrization: $\Om^{\eps}:=\{(\pm x,y)\in\R^2\,|\,(x,y)\in\Om^{\eps}_+\}\cup\Upsilon$. We consider the scattering of the incident piston mode $w^{-}_0$ propagating from $+\infty$ to $-\infty$ by the structure. This leads us to study the problem 
\begin{equation}\label{PbRTNull}
\begin{array}{|rcll}
\Delta u^{\eps} + \lambda u^{\eps} & = & 0 & \mbox{ in }\Om^{\eps}\\[3pt]
 \partial_{\nu^{\eps}}u^{\eps}  & = & 0  & \mbox{ on }\partial\Om^{\eps}\\[3pt]
\multicolumn{4}{|c}{ u^{\eps}-w^{-}_0  \mbox{ is outgoing}}
\end{array}
\end{equation}
with $\lambda\in(0;\pi^2)$. Here $\nu^{\eps}$ stands for the normal unit vector to $\partial\Om^{\eps}$ directed to the exterior of $\Om^{\eps}$. Problem (\ref{PbRTNull}) always admits a solution $u^{\eps}$ (defined up to $\mrm{span}(u_{\mrm{tr}})$ when $\lambda=\lambda^0$). We have the decomposition 
\begin{equation}\label{decomScattering}
u^{\eps}=w_0^{-}+\mathcal{R}(\eps,\lambda)\,w_0^{+}+\tilde{u}^{\eps}\mbox{ for }x\ge d,\qquad u^{\eps}=\mathcal{T}(\eps,\lambda)\,w_0^{-}+\tilde{u}^{\eps}\mbox{ for }x\le -d. 
\end{equation}
In (\ref{decomScattering}), $\tilde{u}^{\eps}$ decays as $O(e^{-\sqrt{\pi^2-\lambda}|x|})$ when $x\to\pm\infty$. Moreover $\mathcal{R}(\eps,\lambda)$ and $\mathcal{T}(\eps,\lambda)$ are the reflection and transmission coefficients. Due to conservation of energy, these complex numbers are related by the formula
\[
|\mathcal{R}(\cdot,\cdot)|^2+|\mathcal{T}(\cdot,\cdot)|^2=1.
\]
In this section, we explain how to use the Fano resonance to find settings (geometry and frequency) such that $\mathcal{R}(\cdot,\cdot)=0$ (non reflection) or $\mathcal{T}(\cdot,\cdot)=0$ (perfect reflection). Physically, when $\mathcal{R}(\cdot,\cdot)=0$ the energy is completely transmitted and the field reflected by the structure is null at infinity as in the reference (straight) geometry. In this situation, the defect in the geometry with respect to the reference setting is invisible for backscattering measurements. On the other hand, when $\mathcal{T}(\cdot,\cdot)=0$, all the energy propagated by the incident field is backscattered as if the waveguide were obstructed.\\
\newline
In order to take advantage of the symmetry of the geometry with respect to the $(Oy)$ axis, we introduce the two half-waveguide problems with different boundary conditions on $\Upsilon$
\begin{equation}\label{PbRTNullHalfWg}
(\mathscr{P}^{N})\ \begin{array}{|rcll}
\Delta v^{\eps} + \lambda v^{\eps} & = & 0 & \mbox{ in }\Om^{\eps}_+\\[3pt]
 \partial_{\nu^{\eps}}v^{\eps}  & = & 0  & \mbox{ on }\partial\Om^{\eps}_+\setminus\overline{\Upsilon}\\[3pt]
 \partial_{x}v^{\eps}  & = & 0  & \mbox{ on }\Upsilon\\[3pt]
\end{array}\qquad(\mathscr{P}^{D})\ \begin{array}{|rcll}
\Delta V^{\eps} + \lambda V^{\eps} & = & 0 & \mbox{ in }\Om^{\eps}_+\\[3pt]
 \partial_{\nu^{\eps}}V^{\eps}  & = & 0  & \mbox{ on }\partial\Om^{\eps}_+\setminus\overline{\Upsilon}\\[3pt]
V^{\eps}  & = & 0  & \mbox{ on }\Upsilon.
\end{array}
\end{equation}
Problems $(\mathscr{P}^{N})$ and $(\mathscr{P}^{D})$ admit respectively solutions which decompose as
\begin{eqnarray}
\label{decompoHalfN}v^{\eps}&=&w_0^{-}+\mathcal{R}^{N}(\eps,\lambda)\,w_0^{+}+\tilde{v}^{\eps}\\[3pt]
\label{decompoHalfD}V^{\eps}&=&w_0^{-}+\mathcal{R}^{D}(\eps,\lambda)\,w_0^{+}+\tilde{V}^{\eps}
\end{eqnarray}
where $\mathcal{R}^{N}(\eps,\lambda)$, $\mathcal{R}^{D}(\eps,\lambda)\in\Cplx$ and where $\tilde{v}^{\eps}$, $\tilde{V}^{\eps}$ decay as $O(e^{-\sqrt{\pi^2-\lambda}x})$ when $x\to+\infty$. Due to conservation of energy, there holds 
\begin{equation}\label{ConservationNRJHalf}
|\mathcal{R}^{N}(\cdot,\cdot)|=|\mathcal{R}^{D}(\cdot,\cdot)|=1.
\end{equation}
By uniqueness of the scattering coefficients in (\ref{decomScattering}),  (\ref{decompoHalfN}), (\ref{decompoHalfD}), we obtain $u^{\eps}(x,y)=(v^{\eps}(x,y)+V^{\eps}(x,y))/2$ in $\Om^{\eps}_+$ and $u^{\eps}(x,y)=(v^{\eps}(-x,y)-V^{\eps}(-x,y))/2$ in $\Om^{\eps}\setminus\overline{\Om^{\eps}_+}$ (up to possible trapped modes). This yields the formulas
\begin{equation}\label{HalfSumDif}
\mathcal{R}(\cdot,\cdot)=\cfrac{1}{2}\,(\,\mathcal{R}^N(\cdot,\cdot)+\mathcal{R}^D(\cdot,\cdot)\,)\qquad\mbox{ and }\qquad\mathcal{T}(\cdot,\cdot)=\cfrac{1}{2}\,(\,\mathcal{R}^N(\cdot,\cdot)-\mathcal{R}^D(\cdot,\cdot)\,).
\end{equation}
Since by assumption trapped modes do not exist for $(\mathscr{P}^{D})$ at $\lambda=\lambda^0$, we can prove that for $\eps>0$ small enough, we have for all $\mu\in[-c\eps^{-1};c\eps^{-1}]$ ($c>0$ is a fixed constant)
\[
|\mathcal{R}^D(\eps,\lambda^0+\eps\ell_m(H)+\eps^2\mu)-
\mathcal{R}^D(0,\lambda^0)|\le C\eps.
\]
On the other hand, according to (\ref{eqn1DAnalysisBis}), $\mu\mapsto\mathcal{R}^N(\eps,\lambda^0+\eps\ell_m(H)+\eps^2\mu)$ rushes along the unit circle for $\mu\in[-c\eps^{-1};c\eps^{-1}]$. As a consequence, if $\mathcal{R}^N(0,\lambda^0)\ne-\mathcal{R}^D(0,\lambda^0)$, thanks to formulas (\ref{ConservationNRJHalf}) which guarantees that $\mathcal{R}^N(\cdot,\cdot)$, $\mathcal{R}^D(\cdot,\cdot)$ are always located on the unit circle, we are sure that for $\eps$ small enough, there is $\lambda^{\eps}$ such that $\mathcal{R}^N(\eps,\lambda^{\eps})=-\mathcal{R}^D(\eps,\lambda^\eps)$. In this case, from (\ref{HalfSumDif}), we deduce that $\mathcal{R}(\eps,\lambda^{\eps})=0$. Analogously, we can get $\mathcal{T}(\eps,\lambda^{\eps})=0$ when $\mathcal{R}^N(0,\lambda^0)\ne\mathcal{R}^D(0,\lambda^0)$. We summarize these results in the following theorem.

\begin{theorem}\label{thmRTNull}
Assume that the profile function $H$ used to define the geometry in (\ref{DefGeom}) satisfies the condition (\ref{AssumptionCoeffs}).\\[3pt]
$\star$ Assume that $\mathcal{R}^N(0,\lambda^0)\ne-\mathcal{R}^D(0,\lambda^0)$. Then for all $\eps\ne0$ small enough, there is $\lambda^{\eps}$ with $|\lambda^{\eps}-(\lambda^0+\eps\ell_m(H))|=O(\eps)$ such that $\mathcal{R}(\eps,\lambda^{\eps})=0$.\\[3pt]
$\star$ Assume that $\mathcal{R}^N(0,\lambda^0)\ne \mathcal{R}^D(0,\lambda^0)$. Then for all $\eps\ne0$ small enough, there is $\lambda^{\eps}$ with $|\lambda^{\eps}-(\lambda^0+\eps\ell_m(H))|=O(\eps^2)$ such that $\mathcal{T}(\eps,\lambda^{\eps})=0$.
\end{theorem}
\begin{remark}
When $\mathcal{R}^N(0,\lambda^0)=-\mathcal{R}^D(0,\lambda^0)$ (resp. $\mathcal{R}^N(0,\lambda^0)=\mathcal{R}^D(0,\lambda^0)$), non reflection (resp. perfect reflection) occurs in $\Om$ for $\lambda=\lambda_0$. However the above analysis does not give a straightforward proof of existence of a pair $(\eps,\lambda^\eps)\ne(0,\lambda_0)$ such that $\mathcal{R}(\eps,\lambda^{\eps})=0$ or $\mathcal{T}(\eps,\lambda^{\eps})=0$.

\end{remark}

\section{Numerical results}\label{SectionNumerics}

In this section, we illustrate the results obtained above. We shall work in two different geometries.

\subsection{$L$-shaped geometry}

For a given $\lambda^0\in(0;\pi^2)$, set $k^0:=\sqrt{\lambda^0}$ and consider the problem $(\mathscr{P}^{N})$ defined in (\ref{PbRTNullHalfWg}) in the half-waveguide $\Om_+=\Om_+(L)$ with 
\begin{equation}\label{GuideEnL}
\Om_+(L)=(0;+\infty)\times(0;1)\cup(0;\pi/k^0)\times(0;L)
\end{equation}
(see Figure \ref{DomainNum} left). In \cite[Thm. 3.1]{ChPaSu}, it is shown that there is an unbounded sequence $(L_n)_n$ such that trapped modes exist in the geometry (\ref{GuideEnL}) for $L=L_n$. Moreover, we have $\lim_{n\to+\infty}|L_{n+1}-L_n|=\pi/k^0$. We assume that $L$ is chosen once for all equal to one of these $L_n$. \\
\newline
We perturb the geometry $\Om_+$ into $\Om^{\eps}_+:=\Om_+(L+\eps)$. Admittedly, this kind of perturbation is not exactly the one considered in \S\ref{paragraphSetting} (see formula (\ref{DefGeom})). However, since there exists an almost identical mapping from $\Om_+$ to $\Om^{\eps}_+$, results are similar. In the following we compute the scattering coefficient $\mathcal{R}^N(\eps,\lambda)$ appearing in (\ref{decompoHalfN}) for several values of $\eps$ and $\lambda\in(0;\pi^2)$. To proceed, we use a $\mrm{P2}$ finite element method in a truncated geometry. On the artificial boundary created by the truncation, a Dirichlet-to-Neumann operator with $20$ terms serves as a transparent condition. 

\begin{figure}[!ht]
\centering
\begin{tikzpicture}[scale=1.4]
\draw[fill=gray!30,draw=none](0,0) rectangle (3,1);
\draw[fill=gray!30,draw=none](0,0) rectangle (1.25,2.5524);
\draw (3,0)--(0,0)--(0,2.5524)--(1.25,2.5524)--(1.25,1)--(3,1); 
\draw[dashed] (3,1)--(3.5,1); 
\draw[dashed] (3,0)--(3.5,0);
\draw[dotted,>-<] (-0.05,2.7)--(1.3,2.7);
\draw[dotted,>-<] (-0.2,-0.05)--(-0.2,2.6);
\node at (0.65,2.85){\small $\pi/k^0$};
\node at (-0.4,1.28){\small $L$};
\end{tikzpicture}\qquad\qquad\includegraphics[width=0.37\textwidth]{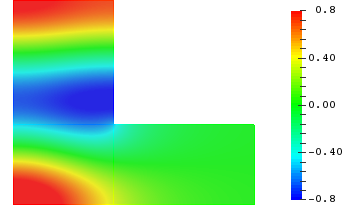}
\caption{Left: geometry of $\Om_+(L)$. Right: real part of a trapped mode appearing for $k^0=\sqrt{\lambda^0}=0.8\pi$ and $L\approx2.5524$.\label{DomainNum}} 
\end{figure}
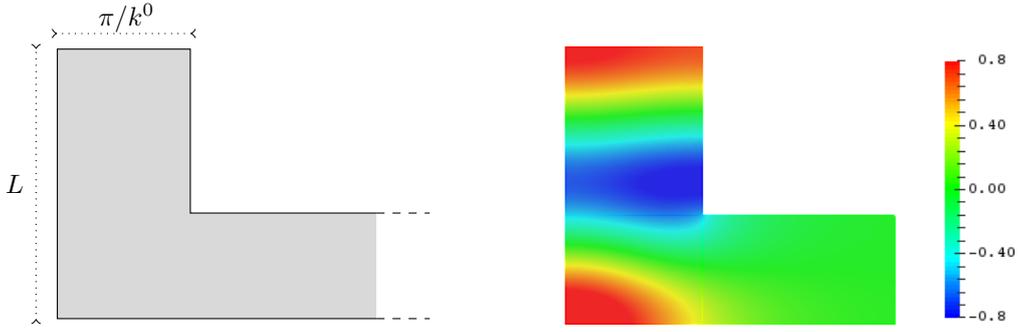

\noindent For the experiments, we take $k^0=\sqrt{\lambda^0}= 0.8\pi$ so that the width of the vertical branch is equal to $1.25$. Figure \ref{DomainNum} right represents the real part of a trapped mode which appears for $L\approx2.5524$. For the justification of existence of this trapped mode, again we refer the reader to \cite[Thm. 3.1]{ChPaSu}. Since $|\mathcal{R}^N(\eps,\lambda)|=1$ due to conservation of energy, we can write $\mathcal{R}^N(\eps,\lambda)=e^{i\theta^N(\eps,\lambda)}$ for some phase function $\theta^N(\cdot,\cdot)$ valued in $[-\pi;\pi)$. In Figure \ref{PhaseL}, we represent the curves $\{\theta^N(\eps,k^2),\,k\in(0;\pi)\}$ for several values of $\eps$. For $\eps\ne0$, we observe a quick variation of the phase for $k$ varying in a neighbourhood of $k^0$. The variation gets even quicker as $\eps$ becomes small. Moreover, we see that the reflection coefficient $\mathcal{R}^N(\eps,\lambda)$ moves on the unit circle counter-clockwise as $\lambda$ increases. These results are coherent with what has been obtained in Section \ref{SectionPerturbed}, especially with formula  (\ref{eqn1DAnalysisBis}).

\begin{figure}[!ht]
\centering
\includegraphics[width=0.47\textwidth]{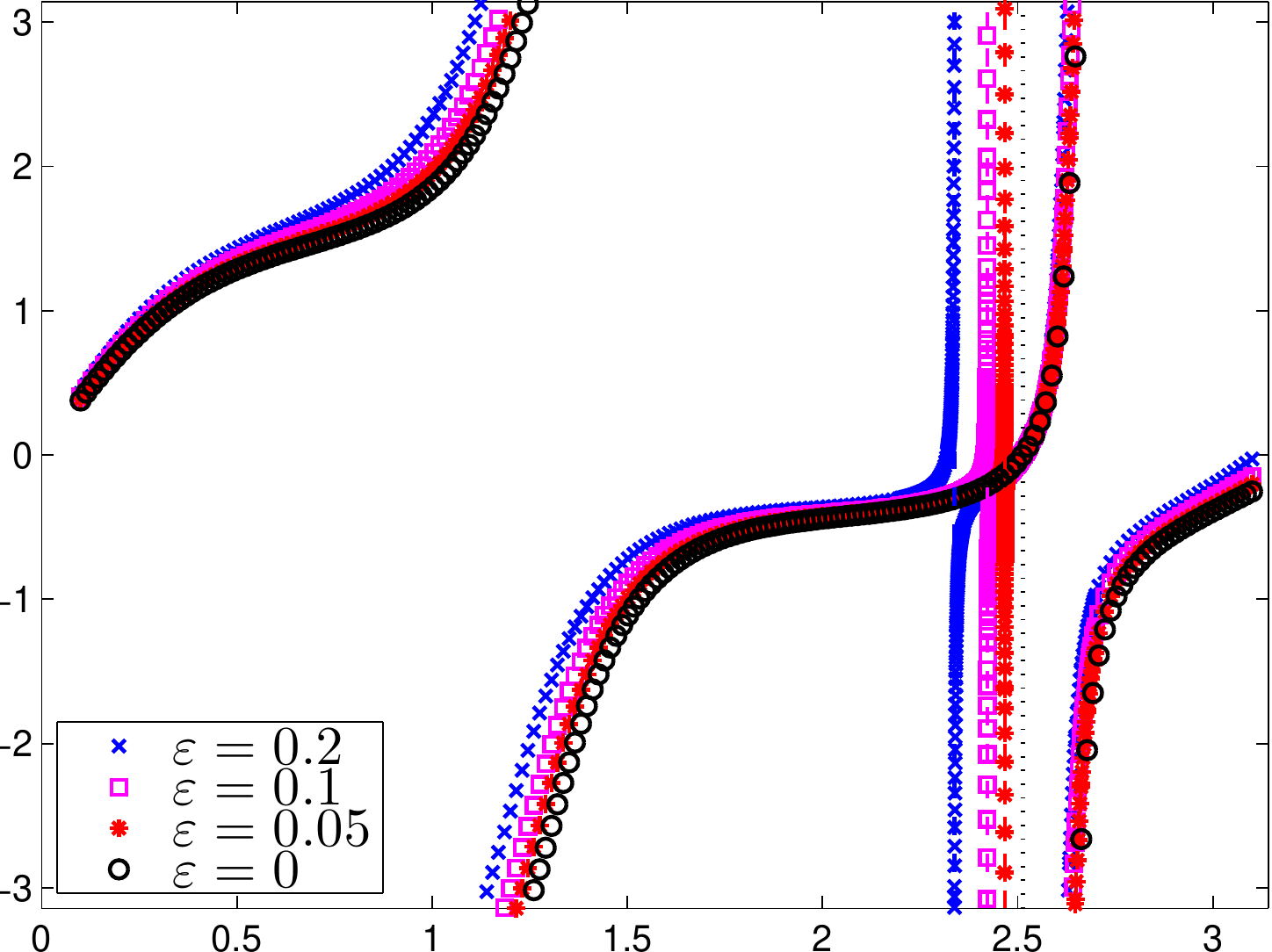}\quad\includegraphics[width=0.47\textwidth]{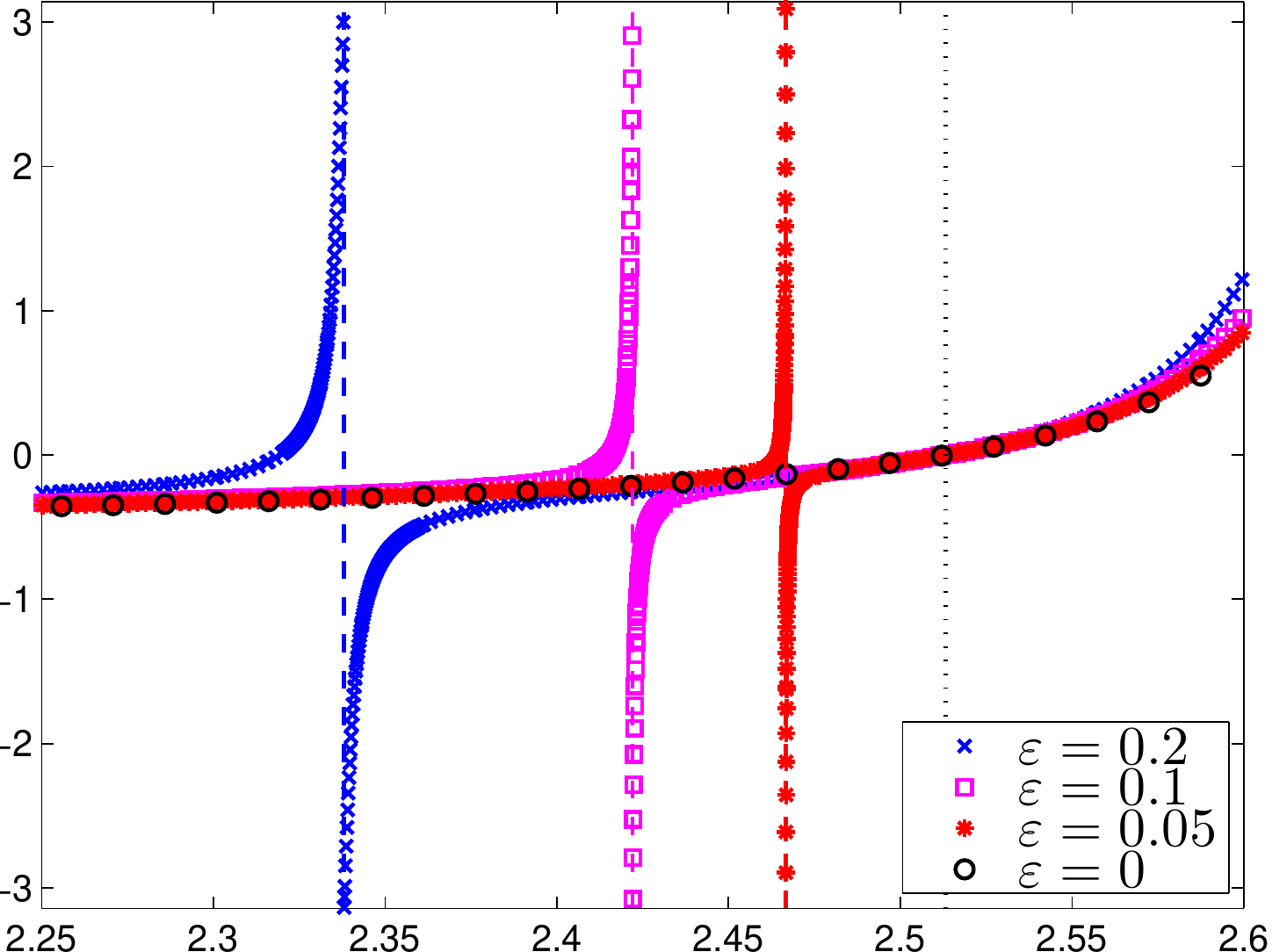}
\caption{Maps $k\mapsto\theta^{N}(\eps,k^2)$ for several values of $\eps$ and $k\in(0;\pi)$. The right picture is a zoom on the left picture around $k^0=\sqrt{\lambda^0}=0.8\pi$ (this value is marked by the vertical black dotted line). The vertical coloured dashed lines indicate the values of $k$ such that $\theta^{N}(\eps,k^2)=-\pi$.  \label{PhaseL}}
\end{figure}

\noindent In the next series of experiments, we work in the geometry $\Om^{\eps}$ obtained from $\Om^{\eps}_+$ by symmetrization: $\Om^{\eps}=\{(\pm x,y)\in\R^2\,|\,(x,y)\in\Om^{\eps}_+\}\cup\Upsilon$ (see the discussion just before (\ref{PbRTNull})). In particular $\Om^{\eps}$ is unbounded both in left and right directions. We take $\eps\ne0$ small (for the experiments, $\eps=0.05$) and we compute the scattering coefficients $\mathcal{R}(\eps,\lambda)$, $\mathcal{T}(\eps,\lambda)$ introduced in (\ref{decomScattering}) for a range of $\lambda$ in a neighbourhood of $\lambda^0$. Our goal is to exhibit situations where $\mathcal{R}(\eps,\lambda)=0$, $|\mathcal{T}(\eps,\lambda)|=1$ (non reflection) or $|\mathcal{R}(\eps,\lambda)|=1$, $\mathcal{T}(\eps,\lambda)=0$ (perfect reflection). To proceed, we select the values of $\lambda$ corresponding to peaks in the curves $\lambda\mapsto-\ln|\mathcal{R}(\eps,\lambda)|$, $\lambda\mapsto-\ln|\mathcal{T}(\eps,\lambda)|$. In Figure \ref{LRNull}, we display the total field $u^{\eps}$ defined in (\ref{decomScattering}) and the scattered field $u^{\eps}-w_0^{-}$ in a situation where $\mathcal{R}(\eps,\lambda)=0$, $|\mathcal{T}(\eps,\lambda)|=1$ (non reflection). We observe that the scattered field is indeed null as $x\to+\infty$. Surprisingly, it seems that the scattered field is also null, or at least small, as $x\to-\infty$. This indicates that not only we have $|\mathcal{T}(\eps,\lambda)|=1$ but also $\mathcal{T}(\eps,\lambda)\approx1$. In Figure \ref{LTNull}, we represent the total field $u^{\eps}$ in a setting where $|\mathcal{R}(\eps,\lambda)|=1$, $\mathcal{T}(\eps,\lambda)=0$ (perfect reflection). As expected the total field is null as $x\to-\infty$. As a consequence, the energy is completely backscattered, this is the mirror effect. It seems that $u^{\eps}$ is purely imaginary as $x\to+\infty$. This means that $\mathcal{R}(\eps,\lambda)=-1$ (or at least $\mathcal{R}(\eps,\lambda)\approx-1$).

\begin{figure}[!ht]
\centering
\includegraphics[width=0.47\textwidth]{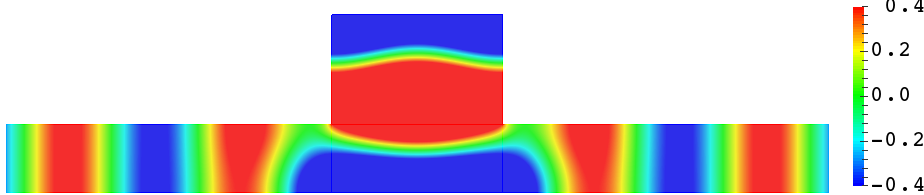}\quad\includegraphics[width=0.47\textwidth]{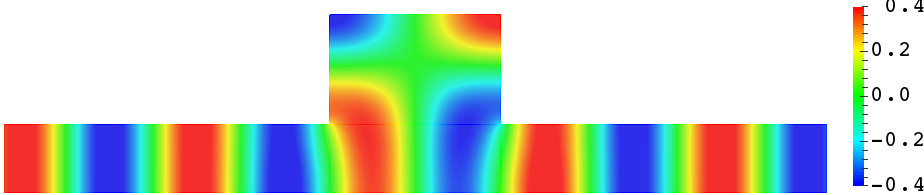}\\[5pt]
\includegraphics[width=0.47\textwidth]{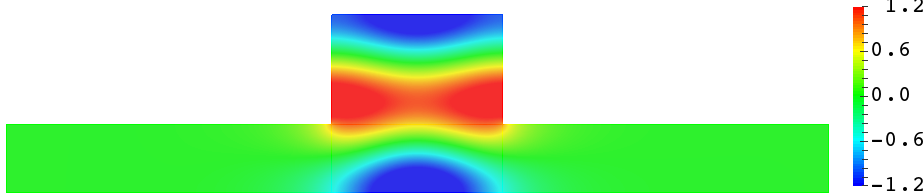}\quad\includegraphics[width=0.47\textwidth]{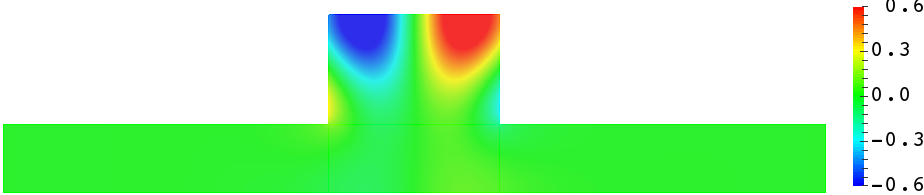}
\caption{Top: real (left) and imaginary (right) parts of the total field $u^{\eps}$. Bottom: real (left) and imaginary (right) parts of the scattered field $u^{\eps}-w_0^{-}$. Here $\eps=0.05$ and $k=\sqrt{\lambda}=2.46402$. The parameter $\lambda$ has been tuned so that $\mathcal{R}(\eps,\lambda)=0$, $|\mathcal{T}(\eps,\lambda)|=1$ (non reflection). As expected the scattered field is null as $x\to+\infty$.\label{LRNull}}
\end{figure}
\begin{figure}[!ht]
\centering
\includegraphics[width=0.47\textwidth]{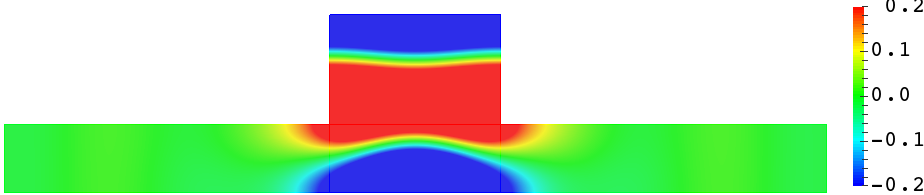}\quad\includegraphics[width=0.47\textwidth]{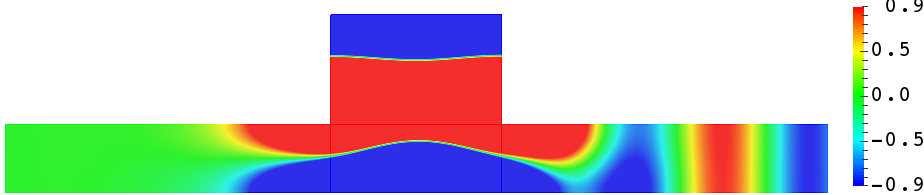}
\caption{Real (left) and imaginary (right) parts of the total field $u^{\eps}$ for $\eps=0.05$ and $k=\sqrt{\lambda}=2.4666602$. The frequency has been tuned so that $|\mathcal{R}(\eps,\lambda)|=1$, $\mathcal{T}(\eps,\lambda)=0$ (perfect reflection). As expected the total field is null as $x\to-\infty$. \label{LTNull} }
\end{figure}

\subsection{Geometry with non penetrable inclusions}

In this paragraph, we make computations in another geometry with non penetrable inclusions. This time, we set 
\[
\Om_+^{\eps}:=(0;+\infty)\times(0;1)\setminus\overline{B((1,0.5+\eps),0.25)}
\]
where $B((1,0.5+\eps),0.25)$ is the open ball centered at $(1,0.5+\eps)$ of radius $0.25$ (see Figure \ref{GeomDisk} left). Note that for $\eps=0$, the half-waveguide is symmetric with respect to the axis $\R\times\{0.5\}$. Classical results (see e.g. \cite{EvLV94}) guarantee the existence of trapped modes in this geometry when $\eps=0$ for certain frequencies. Numerically, we find that they appear for $k^0=\sqrt{\lambda^0}\approx2.7403$. One of theses trapped modes is displayed in Figure \ref{GeomDisk} right. 

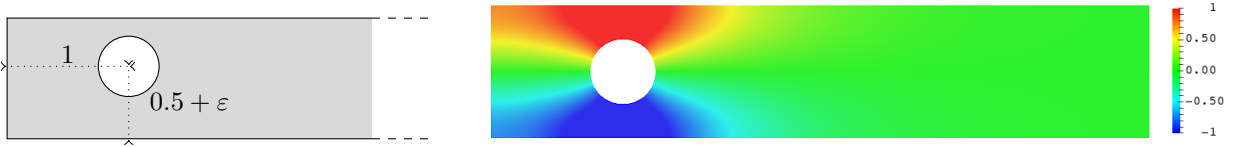
\begin{figure}[!ht]
\centering
\begin{tikzpicture}[scale=1.6]
\draw[fill=gray!30,draw=none](0,0) rectangle (3,1);
\draw (3,0)--(0,0)--(0,1)--(3,1); 
\draw[dashed] (3,1)--(3.5,1); 
\draw[dashed] (3,0)--(3.5,0);
\draw[fill=white] (1,0.6) circle (0.25cm);
\draw[dotted,>-<] (1,-0.05)--(1,0.65);
\draw[dotted,>-<] (-0.05,0.6)--(1.05,0.6);
\node at (1.5,0.3){\small $0.5+\eps$};
\node at (0.5,0.7){\small $1$};
\end{tikzpicture}\qquad\raisebox{0.1cm}{\includegraphics[width=0.58\textwidth]{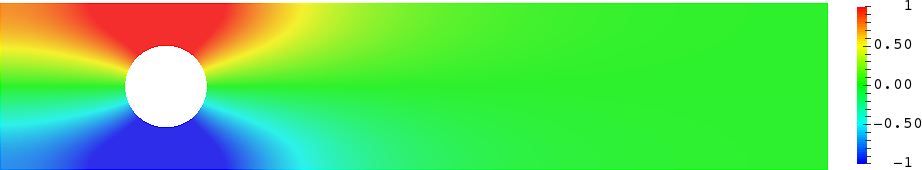}}
\caption{Left: geometry of $\Om_+^{\eps}$. Right: real part of a trapped mode for $\eps=0$ and $k^0:=\sqrt{\lambda^0}\approx2.7403$.\label{GeomDisk}} 
\end{figure}

\noindent We perform the same experiments as in the previous paragraph. In particular, Figures \ref{PhaseDisk}, \ref{CercleRNull}, \ref{CercleTNull} below are respectively the analogous of Figures \ref{PhaseL}, \ref{LRNull}, \ref{LTNull} in the new geometry. Again, we obtain results which are coherent with the theory developed in the previous sections.

\begin{figure}[!ht]
\centering
\includegraphics[width=0.47\textwidth]{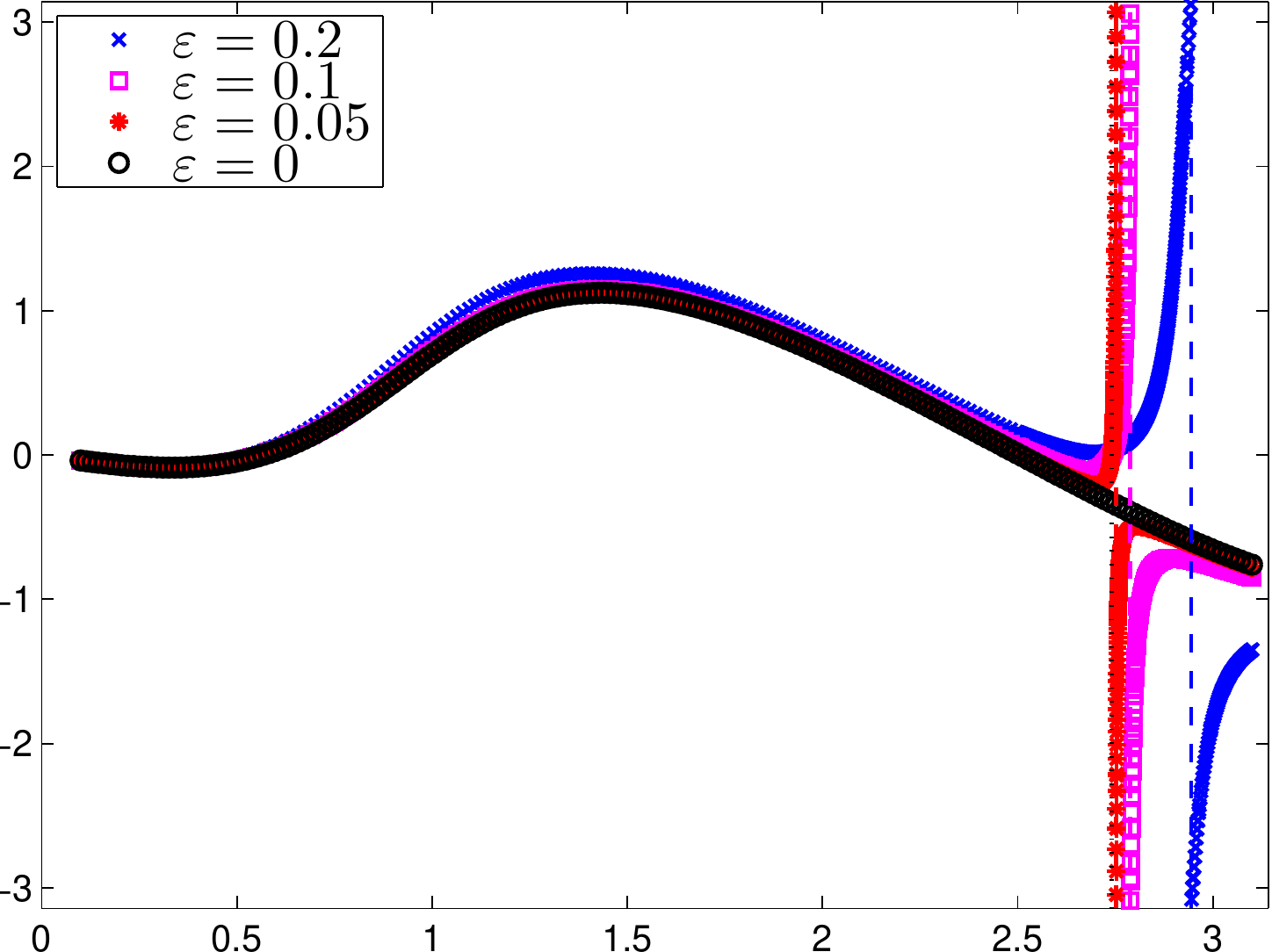}\quad\includegraphics[width=0.47\textwidth]{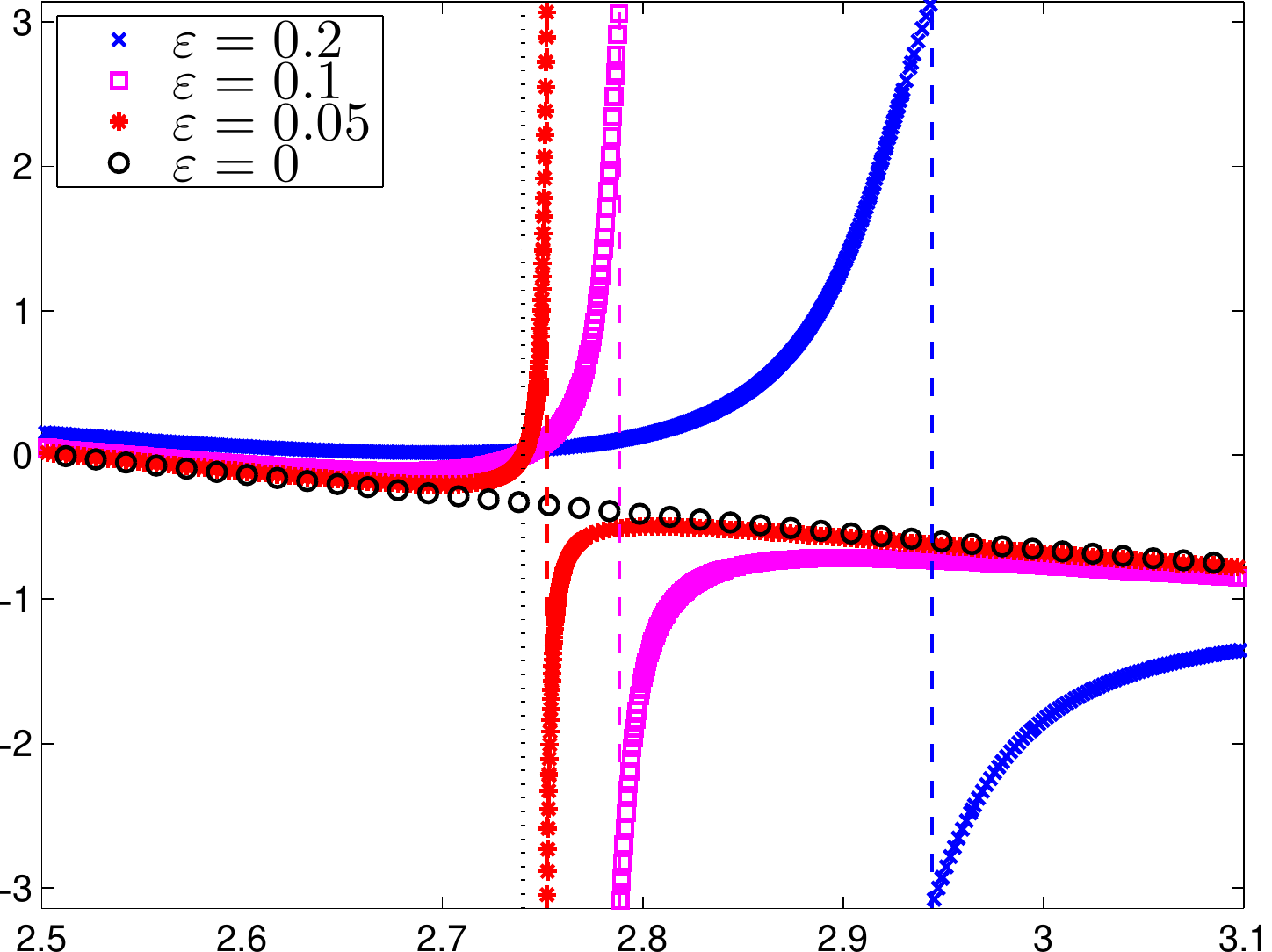}
\caption{Maps $k\mapsto\theta^{N}(\eps,k^2)$ for several values of $\eps$ and $k\in(0;\pi)$. The right picture is a zoom on the left picture around $k^0=\sqrt{\lambda^0}=2.7403$ (this value is marked by the vertical black dotted line). The vertical coloured dashed lines indicate the values of $k$ such that $\theta^{N}(\eps,k^2)=-\pi$.\label{PhaseDisk}}
\end{figure}

\begin{figure}[!ht]
\centering
\includegraphics[width=0.47\textwidth]{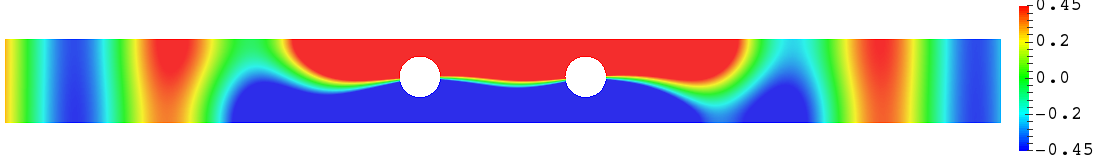}\quad\includegraphics[width=0.47\textwidth]{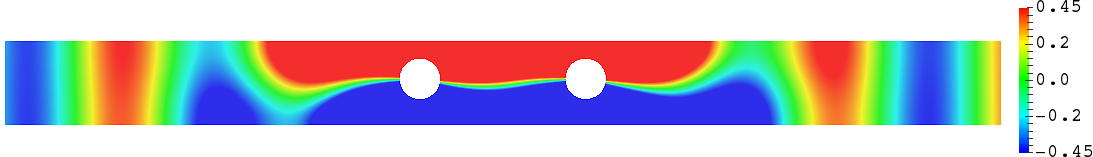}\\[5pt]
\includegraphics[width=0.47\textwidth]{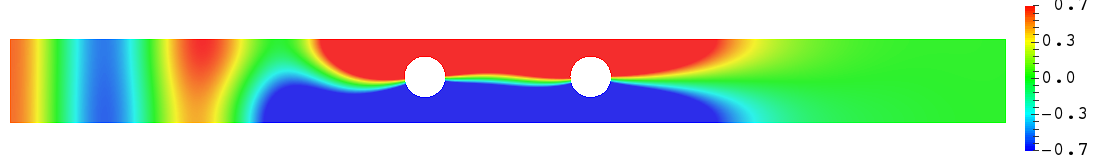}\quad\includegraphics[width=0.47\textwidth]{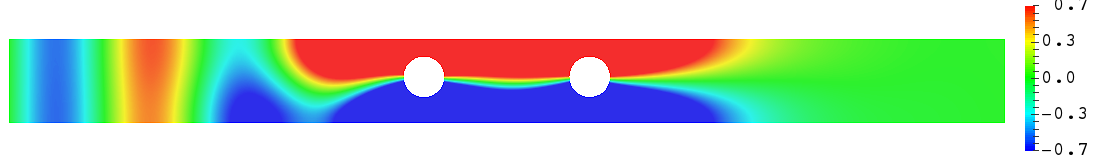}
\caption{Top: real (left) and imaginary (right) parts of the total field $u^{\eps}$. Bottom: real (left) and imaginary (right) parts of the scattered field $u^{\eps}-w_0^{-}$. Here $\eps=0.05$ and $k=\sqrt{\lambda}=2.751$. The parameter $\lambda$ has been tuned so that $\mathcal{R}(\eps,\lambda)=0$, $|\mathcal{T}(\eps,\lambda)|=1$ (non reflection). As expected the scattered field is null as $x\to+\infty$. \label{CercleRNull}}
\end{figure}
\begin{figure}[!ht]
\centering
\includegraphics[width=0.47\textwidth]{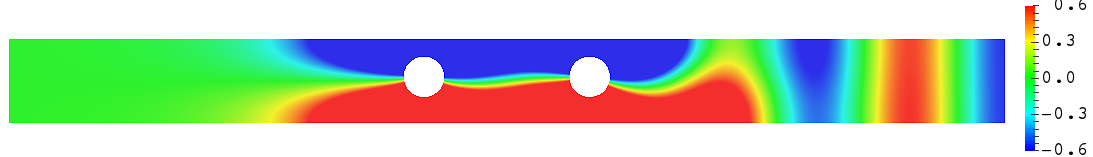}\quad\includegraphics[width=0.47\textwidth]{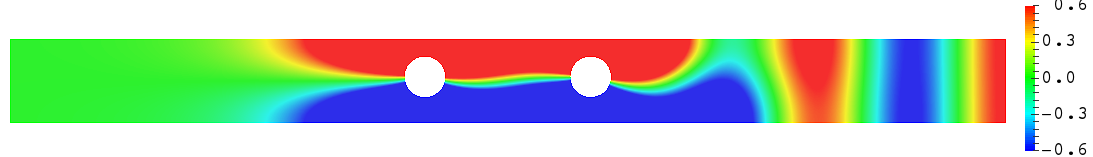}
\caption{Real (left) and imaginary (right) parts of the total field $u^{\eps}$ for $\eps=0.05$ and $k=\sqrt{\lambda}=2.75495$. The frequency has been tuned so that $|\mathcal{R}(\eps,\lambda)|=1$, $\mathcal{T}(\eps,\lambda)=0$ (perfect reflection). As expected the total field is null as $x\to-\infty$.\label{CercleTNull}}
\end{figure}

\newpage

\section{Conclusion}
In this article, we proved rigorously the Fano resonance phenomenon in a $\mrm{2D}$ waveguide with homogeneous Neumann boundary conditions. Then in monomode regime, we used this Fano resonance phenomenon together with symmetry considerations to provide examples of geometries where the energy of an incident wave propagating through the structure is completely transmitted ($\rcoef=0$) or completely reflected ($\tcoef=0$). Everything presented here can be adapted in higher dimension, that is for waveguides of $\R^d$, $d\ge3$, with a bounded transverse section. Moreover we could also have studied analogously problems with homogeneous Dirichlet boundary conditions to deal with quantum waveguides. We considered a geometrical perturbation of the walls of the waveguide. We could also have worked with a penetrable inclusion placed in the waveguide. Then perturbing the material parameter, we would have obtained similar results. The explanation of the Fano resonance phenomenon is valid at any frequency in any waveguide. However the results of existence of geometries where non reflection or perfect reflection occurs works only at low frequency, $0<k<\pi$ in our geometry so that only the piston modes can propagate, and in symmetric domains. It seems rather intricate to extent our approach at higher frequency when more modes can propagate and to remove the symmetry assumption. It would also be interested to investigate the Fano resonance phenomenon around an eigenvalue of higher multiplicity (in this work we assume that the eigenvalue is simple).

\section*{Acknowledgments}
The research of S.A. N. was supported by the grant No. 17-11-01003 of the Russian Science Foundation.

\bibliography{Bibli}
\bibliographystyle{plain}

\end{document}